\documentclass[12pt]{amsart}
\usepackage{amsthm}
\usepackage{amssymb}
\usepackage{amsmath}
\usepackage{color}

\makeatletter

\theoremstyle{plain}
\newtheorem{proposition}{Proposition}
\newtheorem{theorem}[proposition]{Theorem}
\newtheorem{lemma}[proposition]{Lemma}
\newtheorem{corollary}[proposition]{Corollary}

\theoremstyle{definition}
\newtheorem{definition}[proposition]{Definition}

\theoremstyle{definition}

\newtheorem{remark}[proposition]{Remark}

\numberwithin{equation}{section}

\numberwithin{proposition}{section}

\gdef\myletter{}

\let\savetheequation\theequation
\def\theequation{\savetheequation\myletter}

\usepackage{color}
\def\blue{\color{blue}}

\newcommand{\CC}{{\mathbb C}}

\newcommand{\RR}{{\mathbb R}}
\newcommand{\ZZ}{{\mathbb Z}}

\newcommand{\PP}{{\mathbb P}}

\newcommand{\Ec}{\mathcal{E}}
\newcommand{\capacity}{\textup{Cap}}

\renewcommand{\Im}{\mbox{Im}}

\renewcommand{\Re}{\mbox{Re}}

\def \bar{\overline}

\author[T. Bayraktar]{Turgay Bayraktar}
\address{Sabanci University, Istanbul, Turkey}
\email{tbayraktar@sabanciuniv.edu}

\author[T. Bloom]{Thomas Bloom}
\address{University of Toronto, Toronto, Ontario M5S 2E4 Canada}
\email{bloom@math.toronto.edu}

\author[N. Levenberg]{Norman Levenberg}
\address{Indiana University, Bloomington, IN 47405 USA}
\email{nlevenbe@indiana.edu}

\author[C.H. Lu]{Chinh H. Lu}
\address{Universit\'e Paris-Sud, Orsay, France, 91405}
\email{hoang-chinh.lu@u-psud.fr}

\title[Convex Bodies]{\bf Pluripotential Theory and Convex Bodies: Large Deviation Principle}
\date{\today}

\thanks{N. Levenberg is supported by Simons Foundation grant No. 354549}


\usepackage{hyperref}
\hypersetup{
    bookmarks=true,         
    unicode=false,          
    pdftoolbar=true,       
    pdfmenubar=true,       
    pdffitwindow=false,    
    pdfstartview={FitH},   
    pdftitle={Pluripotential Theory and Convex Bodies},    
    pdfauthor={Bayraktar, Bloom, Levenberg, Lu},     
    colorlinks=true,       
   linkcolor=black,          
    citecolor=black,        
    filecolor=black,      
    urlcolor=black}           

\begin{document}

\maketitle

\begin{abstract} We continue the study in \cite{BBL} in the setting of weighted pluripotential theory arising from polynomials associated to a convex body $P$ in $(\RR^+)^d$. Our goal is to establish a large deviation principle in this setting specifying the rate function in terms of $P-$pluripotential-theoretic notions. As an important preliminary step, we first give an  existence proof for the solution of a Monge-Amp\`ere equation in an appropriate finite energy class. This is achieved using a variational approach.
\end{abstract}

\tableofcontents

\section{Introduction}\label{sec:intro} As in \cite{BBL}, we fix a convex body $P\subset (\RR^+)^d$ and we define the logarithmic indicator function
\begin{equation}\label{logind}H_P(z):=\sup_{J\in P} \log |z^J|:=\sup_{(j_1,...,j_d)\in P} \log[|z_1|^{j_1}\cdots |z_d|^{j_d}].\end{equation}
We assume throughout that 
\begin{equation}\label{sigmainkp} \Sigma \subset kP \ \hbox{for some} \ k\in \ZZ^+  \end{equation}
where 
$$\Sigma:=\{(x_1,...,x_d)\in \RR^d: 0\leq x_i \leq 1, \ \sum_{j=1}^d x_i \leq 1\}.$$ 
Then
$$H_P(z)\geq \frac{1}{k}\max_{j=1,...,d}\log^+ |z_j| $$
where $\log^+ |z_j| =\max[0,\log|z_j|]$. We define
$$L_P=L_P(\CC^d):= \{u\in PSH(\CC^d): u(z)- H_P(z) =0(1), \ |z| \to \infty \},$$ and 
$$L_{P,+}=L_{P,+}(\CC^d)=\{u\in L_P(\CC^d): u(z)\geq H_P(z) + C_u\}.$$
These are generalizations of the classical Lelong classes when $P=\Sigma$. We define the finite-dimensional polynomial spaces 
$$Poly(nP):=\{p(z)=\sum_{J\in nP\cap (\ZZ^+)^d}c_J z^J: c_J \in \CC\}$$
for $n=1,2,...$ where $z^J=z_1^{j_1}\cdots z_d^{j_d}$ for $J=(j_1,...,j_d)$.
For $p\in Poly(nP), \ n\geq 1$ we have $\frac{1}{n}\log |p|\in L_P$; also each $u\in L_{P,+}(\CC^d)$ is locally bounded in $\CC^d$. For $P=\Sigma$, we write $Poly(nP)=\mathcal P_n$.

Given a compact set $K\subset \CC^d$, one can define various pluripotential-theoretic notions associated to $K$ related to $L_P$ and the polynomial spaces $Poly(nP)$. Our goal in this paper is to prove some probabilistic properties of random point processes on $K$ utilizing these notions and their weighted counterparts. We require an existence proof for the solution of a Monge-Amp\`ere equation in an appropriate finite energy class; this is done in Theorem \ref{thm: existence in E1} using a variational approach and is of interest on its own. The third section recalls appropriate definitions and properties in $P-$pluripotential theory, mostly following \cite{BBL}. Subsection 3.3 includes a standard elementary probabilistic result on almost sure convergence of probability measures associated to random arrays on $K$ to a $P-$pluripotential-theoretic equilibrium measure. Section 4 sets up the machinery for the more subtle large deviation principle (LDP), Theorem \ref{ldp}, for which we provide two proofs (analogous to those in \cite{PELD}). 

\section{Monge-Amp\`ere and $P-$pluripotential theory}

\subsection{Monge-Amp\`ere equations with prescribed singularity}

In this section, $(X,\omega)$ is a compact K\"ahler manifold of dimension
 $d$. 
 
 \subsubsection{Quasi-plurisubharmonic functions}
 
 A function $u: X \rightarrow \mathbb{R}\cup \{-\infty\}$ is called quasi-plurisubharmonic (quasi-psh) if locally $u=\rho + \varphi$, where $\varphi$ is plurisubharmonic and $\rho$ is smooth. 
 
 We let $PSH(X,\omega)$ denote the set of $\omega$-psh functions, i.e. quasi-psh functions $u$ such that $\omega_u:= \omega+dd^c u\geq 0$ in the sense of currents on $X$.

 Given $u,v\in PSH(X,\omega)$ we say that $u$ is more singular than $v$ (and we write $u\prec v$) if $u\leq v+C$ on $X$, for some constant $C$. We say that $u$ has the same singularity as $v$ (and we write $u\simeq v$) if $u\prec v$ and $v\prec u$. 
 
 Given $\phi \in PSH(X,\omega)$, we let $PSH(X,\omega,\phi)$ denote the set of $\omega$-psh functions $u$ which are more singular than $\phi$.  
 
 \subsubsection{Nonpluripolar Monge-Amp\`ere measure}
 For  bounded $\omega$-psh functions $u_1,...,u_d$, the Monge-Amp\`ere product $(\omega+dd^c u_1)\wedge ... \wedge (\omega+dd^c u_d)$ is well-defined as a positive Radon measure on $X$ (see \cite{GZ05}, \cite{BT}).  For general $\omega$-psh functions $u_1,...,u_d$, the sequence of positive measures 
 $$
 {\bf 1}_{\cap \{u_j >-k\}} (\omega+dd^c \max(u_1,-k))\wedge... \wedge (\omega+dd^c \max(u_d,-k)) 
 $$
 is non-decreasing in $k$ and the limiting measure, which is called the nonpluripolar product of $\omega_{u_1},...,\omega_{u_d}$, is denoted by 
$$
\omega_{u_1} \wedge ... \wedge \omega_{u_d}. 
$$ 
 When $u_1=...=u_d=u$ we write $\omega_u^d: =\omega_u \wedge ...\wedge \omega_u$.   Note that by definition $\int_X \omega_{u_1}\wedge ... \wedge \omega_{u_d} \leq \int_X \omega^d$. 
 
  It was proved in \cite[Theorem 1.2]{WN} and \cite[Theorem 1.1]{DDL} that the total mass of nonpluripolar Monge-Amp\`ere products is decreasing with respect to singularity type. More precisely, 
 \begin{theorem}
 	\label{thm: mass monotone}
 	Let $\omega_1,...,\omega_d$ be K\"ahler forms on $X$. If $u_j\leq v_j$, $j=1,...,d$, are $\omega_j$-psh functions then 
 	$$
 	\int_X (\omega_1+ dd^c {u_1}) \wedge ... \wedge (\omega_d+dd^c {u_d}) \leq \int_X (\omega_1+dd^c {v_1})\wedge ... \wedge (\omega_d+dd^c {v_d}).
 	$$
 \end{theorem}

As noted above, for a general $\omega$-psh function $u$ we  have the estimate $\int_X \omega_u^d \leq \int_X \omega^d$. Following \cite{GZ07} we let $\mathcal{E}(X,\omega)$ denote the set of all $\omega$-psh functions with maximal total mass, i.e. 
  $$
  \mathcal{E}(X,\omega) := \left \{ u \in PSH(X,\omega) : \int_X \omega_u^d =\int_X \omega^d \right \}.
  $$
Given $\phi \in PSH(X,\omega)$, we define 
  $$\mathcal{E}(X,\omega,\phi) := \left \{u \in PSH(X,\omega,\phi) \; : \; \int_X \omega_u^d =\int_X \omega_{\phi}^d \right\}.$$ 
\begin{proposition}
Let $\phi \in PSH(X,\omega)$. The following are equivalent :
\begin{enumerate}
	\item $\Ec(X,\omega,\phi) \cap \Ec(X,\omega) \neq \emptyset$;
	\item $\phi \in \Ec(X,\omega)$; 
	\item $\Ec(X,\omega,\phi) \subset \Ec(X,\omega)$.
\end{enumerate}
\end{proposition}
\begin{proof}
We first prove $(1) \Longrightarrow (2)$. If $u \in \Ec(X,\omega,\phi)\cap \Ec(X,\omega)$ then $\int_X \omega_u^d = \int_X \omega^d$. On the other hand, since $u$ is more singular than $\phi$, Theorem \ref{thm: mass monotone} ensures that
	$$
	\int_X \omega^d = \int_X \omega_u^d \leq \int_X \omega_{\phi}^d \leq \int_X \omega^d,
	$$
	hence equality holds, proving that $\phi \in \Ec(X,\omega)$. 
	
	Now we prove $(2)\Longrightarrow (3)$. 
	If $\phi\in \Ec(X,\omega)$ and $u\in \Ec(X,\omega,\phi)$ then 
	$$\int_X \omega_u^d =\int_X \omega_{\phi}^d =\int_X \omega^d,$$
	hence $u\in \Ec(X,\omega)$. 
	
	Finally $(3) \Longrightarrow (1)$ is obvious. 
\end{proof}

\begin{proposition}\label{prop: mixed MA mass global}
Assume that $\phi_j \in PSH(X,\omega_j)$, $j=1,...,d$ with $\int_X (\omega_j+dd^c \phi_j)^d>0$.  If $u_j \in \mathcal{E}(X,\omega_j,\phi_j)$, $j=1,...,d$, then
$$
\int_X (\omega_1+dd^c u_1) \wedge ... \wedge (\omega_d+dd^c u_d) = \int_X (\omega_1+dd^c \phi_1) \wedge ... \wedge (\omega_d+dd^c \phi_d). 
$$ 
\end{proposition}

\begin{proof}
Theorem \ref{thm: mass monotone} gives one inequality. The other one follows from \cite[Proposition 3.1 and Theorem 3.14]{DDL}.
\end{proof}

\subsubsection{Model potentials}

For a function $f: X\rightarrow \mathbb{R}\cup \{-\infty\}$, we let $f^*$ denote its uppersemicontinuous (usc) regularization, i.e. 
 $$
 f^*(x) := \limsup_{X\ni y \to x} f(y). 
 $$
 Given $\phi \in PSH(X,\omega)$,  following J. Ross and D. Witt Nystr\"om \cite{RWN},  we define 
 $$
 P_{\omega}[\phi] := \left(\lim_{t\to +\infty} P_{\omega}(\min(\phi+t, 0))\right)^*.  
 $$
 Here, for a function $f$, $P_{\omega}(f)$ is defined as 
 $$
 P_{\omega}(f) : =\left (x\mapsto \sup \{u(x) : u\in PSH(X,\omega), u\leq f \}\right)^*. 
 $$
 It was shown in \cite[Theorem 3.8]{DDL} that the nonpluripolar  Monge-Amp\`ere measure of $P_{\omega}[\phi]$ is dominated by Lebesgue measure:  
 \begin{equation}\label{thm3.8}
 (\omega+dd^c P_{\omega}[\phi])^d \leq {\bf 1}_{\{P_{\omega}[\phi]=0\}}  \omega^d \leq \omega^d. 
\end{equation}
 This fact plays a crucial role in solving the complex Monge-Amp\`ere equation. For the reader's convenience, we note that in the notation of \cite{DDL} (on the left) 
 $$P_{[\omega,\phi]}(0)=P_{\omega}[\phi].$$
 
 \begin{definition}\label{modelstuff}
 	A function $\phi\in PSH(X,\omega)$ is called a model potential if $\int_X \omega_{\phi}^d>0$ and $P_{\omega}[\phi]=\phi$.  A  function $u\in PSH(X,\omega)$ has model type singularity if $u$ has the same singularity as  $P_{\omega}[u]$; i.e., $u- P_{\omega}[u]$ is bounded on $X$.
 \end{definition}

 There are plenty of model potentials. If $\varphi\in PSH(X,\omega)$ with $\int_X \omega_{\varphi}^d >0$ then, by \cite[Theorem 3.12]{DDL}, $P_{\omega}[\varphi]$ is a model potential. In particular, if $\int_X \omega_{\varphi}^d =\int_X \omega^d$ (i.e. $\varphi \in \mathcal{E}(X,\omega)$) then $P_{\omega}[\varphi]=0$.

 We will use the following property of model potentials proved in \cite[Theorem 3.12]{DDL}: if $\phi$ is a model potential then 
\begin{equation}
	\label{sup model potential}
	u \in PSH(X,\omega,\phi)  \Longrightarrow u-\sup_X u \leq \phi. 
\end{equation}

%
 
 In the sequel we always assume that  $\phi$ has {\it model type singularity} and {\it small unbounded locus}; { i.e., $\phi$ is locally bounded outside a closed complete pluripolar set,} allowing us to use the variational approach of \cite{BBGZ} as explained in \cite{DDL}.
 
 \subsubsection{The variational approach}
 
We call a measure which puts no mass on pluripolar sets a {\it nonpluripolar measure}. For a positive nonpluripolar measure $\mu$ on $X$ we let $L_{\mu}$ denote 
the following linear functional on $PSH(X,\omega,\phi)$: 
$$
L_{\mu}(u) := \int_X (u-\phi) d\mu. 
$$

For $u \in PSH(X,\omega)$ with $u \simeq \phi$, we define the Monge-Amp\`ere energy
\begin{equation}\label{MAEN}
{\bf E}_{\phi}(u) := \frac{1}{(d+1)} \sum_{k=0}^d \int_X (u-\phi) \omega_{u}^k \wedge \omega_{\phi}^{d-k}.\end{equation}
  It was shown in \cite[Theorem 4.10]{DDL} (by adapting the arguments of \cite{BBGZ}) that ${\bf E}_{\phi}$ is non-decreasing and concave along affine curves, giving rise to its trivial extension to $PSH(X,\omega,\phi)$. 

We define 
\begin{equation}\label{E1eqn}
\mathcal{E}^1(X,\omega,\phi):= \{u \in PSH(X,\omega,\phi) : {\bf E}_{\phi}(u) >-\infty \}. 
\end{equation}
The following criterion was proved in  \cite[Theorem 4.13]{DDL}: 
\begin{proposition}\label{prop: E1 characterization}
Let $u \in PSH(X,\omega,\phi)$. Then $u\in \mathcal{E}^1(X,\omega,\phi)$ iff $u\in \mathcal{E}(X,\omega,\phi)$ and $\int_X (u-\phi) \omega_u^d >-\infty$. 	
\end{proposition}

 \begin{lemma}
 	\label{lem: pluripolar}
 	If $E$ is pluripolar then there exists $u\in \Ec^1(X,\omega,\phi)$ such that $E\subset \{u=-\infty\}$. 
 \end{lemma}
 
 \begin{proof}
  Without loss of generality we can assume that $\phi$ is a model potential. Then \eqref{thm3.8} gives $\int_X |\phi| \omega_{\phi}^d=0$.  It follows from \cite[Corollary 2.11]{BBGZ} that there exists $v\in \Ec^1(X,\omega,0)$, $v\leq 0$,  such that $E\subset \{v=-\infty\}$. Set $u:= P_{\omega}(\min(v,\phi))$. Then $E\subset \{u=-\infty\}$ and we claim that $u\in \Ec^1(X,\omega,\phi)$. For each $j\in \mathbb{N}$ we set $v_j := \max(v,-j)$ and $u_j := P_{\omega}(\min(v_j,\phi))$. Then $u_j$ decreases to $u$ and $u_j \simeq \phi$.  Using \cite[Theorem 4.10 and Lemma 4.15]{DDL} it suffices to check that $\{\int_X |u_j-\phi| \omega_{u_j}^d\}$ is uniformly bounded. 
   It follows from \cite[Lemma 3.7]{DDL} that  
  \begin{flalign*}
  \int_X |u_j-\phi| \omega_{u_j}^d  \leq \int_X |u_j| \omega_{u_j}^d & \leq \int_X |v_j|  \omega_{v_j}^d + \int_X |\phi| \omega_{\phi}^d\\
  & =  \int_X |v_j|  \omega_{v_j}^d.
  \end{flalign*}
  The fact that $\int_X |v_j| \omega_{v_j}^d$ is uniformly bounded follows from \cite[Corollary 2.4]{GZ07} since $v\in \Ec^1(X,\omega,0)$. This concludes the proof. 
 \end{proof}

\begin{lemma}
	\label{lem: Lmu bounded}
	Assume that $\mathcal{E}^1(X,\omega,\phi) \subset L^1(X,\mu)$. Then, for each $C>0$,  $L_{\mu}$ is bounded  on 
	$$
	E_C := \{u \in PSH(X,\omega,\phi) : \sup_X u \leq 0 \ \text{and}\  {\bf E}_{\phi}(u) \geq -C \}.
	$$
\end{lemma}

\begin{proof}
By concavity of ${\bf E}_{\phi}$ the set $E_C$ is convex. We now show that $E_C$ is compact in the $L^1(X,\omega^d)$ topology.  Let $\{u_j\}$ be a sequence in $E_C$.  We claim that $\{\sup_X u_j\}$ is bounded. Indeed, by \cite[Theorem 4.10]{DDL} 
$$
{\bf E}_{\phi}(u_j) \leq \int_X (u_j-\phi) \omega_{\phi}^d $$
$$\leq (\sup_X u_j) \int_X \omega_{\phi}^d + \int_X (u_j-\sup_X u_j -\phi) \omega_{\phi}^d. 
$$
It follows from \eqref{sup model potential} that $u_j-\sup_X u_j \leq P_{\omega}[\phi] \leq \phi +C_0$, where $C_0$ is a constant. The boundedness of $\{\sup_X u_j\}$ then follows from that of $\{{\bf E}_{\phi}(u_j)\}$ and  the above estimate. This proves the claim.  

A subsequence of $\{u_j\}$, still denoted by $\{u_j\}$, converges in $L^1(X,\omega^d)$ to $u\in PSH(X,\omega)$ with $\sup_X u \leq 0$.  Since $u_j-\sup_X u_j \leq \phi+C_0$, we have $u-\sup_X u\leq \phi+C_0$. This proves that $u\in PSH(X,\omega,\phi)$. The upper semicontinuity of ${\bf E}_{\phi}$ (see \cite[Proposition 4.19]{DDL}) ensures that ${\bf E}_{\phi}(u) \geq -C$, hence $u\in E_C$. This proves that $E_C$ is compact in the $L^1(X,\omega^d)$ topology.

The result then follows from
  \cite[Proposition 3.4]{BBGZ}. 	
  \end{proof}

 The goal of this section is to prove the following result: 
 \begin{theorem}
 	\label{thm: existence in E1}
 	Assume that $\mu$ is a nonpluripolar positive measure on $X$ such that $\mu(X)=\int_X \omega_{\phi}^d$. The following are equivalent
 	\begin{enumerate}
 		\item $\mu$ has finite energy, i.e., $L_{\mu}$ is finite on $\mathcal{E}^1(X,\omega,\phi)$; 
 		\item there exists  $u\in \mathcal{E}^1(X,\omega,\phi)$ such that $\omega_u^d=\mu$;
 		\item there exists a unique $u\in \mathcal{E}^1(X,\omega,\phi)$ such that $$F_{\mu}(u) = \max_{v \in \Ec^1(X,\omega,\phi)} F_{\mu}(v)<+\infty$$
		{where $F_{\mu}={\bf E}_{\phi}-L_{\mu}$.} 
 	\end{enumerate}
 	 \end{theorem}
 
\begin{remark}
	It was shown in \cite[Theorem 4.28]{DDL} 
	that a unique (normalized) solution $u$ in $\mathcal{E}(X,\omega,\phi)$ always exists
	 (without the finite energy assumption on $\mu$). But that proof does not give a solution in $\mathcal{E}^1(X,\omega,\phi)$. Below, we will follow the proof of  \cite[Theorem 4.28]{DDL}  and use the finite energy condition, $\mathcal{E}^1(X,\omega,\phi) \subset L^1(X,\mu)$, to prove that $u$ 
	  belongs to $\mathcal{E}^1(X,\omega,\phi)$. 
\end{remark}

\begin{lemma}
	\label{lem: properness of Lmu}
	Assume that $\mathcal{E}^1(X,\omega,\phi) \subset L^1(X,\mu)$. Then there exists a positive constant $C$ such that, for all $u\in \mathcal{E}^1(X,\omega,\phi)$ with $\sup_X u =0$, 
	\begin{equation}
		\label{eq: F mu proper}
		L_{\mu}(u) \geq -C( 1+|{\bf E}_{\phi}(u)|^{1/2}).
	\end{equation}
\end{lemma}
The proof below uses ideas in \cite{GZ07,BBGZ}. 
\begin{proof}
	Since $\phi$ has model type singularity, it follows from \cite[Theorem 4.10]{DDL} that ${\bf E}_{\phi}-{\bf E}_{P_{\omega}[\phi]}$ is bounded. Without loss of generality we can assume in this proof that $\phi=P_{\omega}[\phi]$.  Fix $u\in \mathcal{E}^1(X,\omega,\phi)$ such that $\sup_X u = 0$ and $|{\bf E}_{\phi}(u)|>1$. Then, by \cite[Theorem 3.12] {DDL},  $u\leq \phi$.	Set $a = |{\bf E}_{\phi}(u)|^{-1/2} \in (0,1)$, and $v:= au + (1-a) \phi \in \mathcal{E}^1(X,\omega,\phi)$. We estimate ${\bf E}_{\phi}(v)$ as follows
\begin{flalign*}
	(d+1) {\bf E}_{\phi}(v) &=  a\sum_{k=0}^d \int_X (u-\phi) \omega_{v}^k \wedge \omega_{\phi}^{d-k}\\
	& = a \sum_{k=0}^d  \int_X (u-\phi) (a\omega_u+(1-a) \omega_{\phi})^{k} \wedge \omega_{\phi}^{d-k}\\
	& \geq C(d) a \int_X (u-\phi) \omega_{\phi}^d + C(d) a^2\sum_{k=0}^{d} \int_X (u-\phi) \omega_u^{k} \wedge \omega_{\phi}^d,
\end{flalign*}
where $C(d)$ is  a positive constant which only depends on $d$.  It follows from $\phi=P_{\omega}[\phi]$ and \cite[Theorem 3.8]{DDL} that $\omega_{\phi}^d \leq \omega^d$ (recall (\ref{thm3.8})). This together with \cite[Proposition 2.7]{GZ05} give
$$
\int_X (u-\phi) \omega_{\phi}^d  \geq -C_1, 
$$ 
for a uniform constant $C_1$. Therefore, 
$$
(d+1) {\bf E}_{\phi}(v) \geq -C_1 C(d) a + C_2 a^2 {\bf E}_{\phi}(u) \geq -C_3. 
$$
It thus follows from Lemma \ref{lem: Lmu bounded} that $L_{\mu}(v) \geq -C_4$ for a uniform constant $C_4>0$.  Thus
$$
\int_X (u-\phi)   d\mu \geq -C_4 / a,
$$
which gives \eqref{eq: F mu proper}.
\end{proof}

 We are now ready to prove Theorem \ref{thm: existence in E1}. 
 
 \begin{proof}[Proof of Theorem \ref{thm: existence in E1}]
 Without loss of generality we can assume that $\phi$ is a model potential. 
 We first prove $(1) \Longrightarrow (2)$. 
 We write $\mu = f\nu$, where $\nu$ is a nonpluripolar positive measure satisfying, for all Borel subsets $B\subset X$, 
$$
\nu(B) \leq A \capacity_{\phi}(B),
$$
for some positive constant $A$, and $0\leq f\in L^1(X,\nu)$ {(cf., \cite[Lemma 4.26]{DDL})}. Here $\capacity_{\phi}$ is defined as 
$$
\capacity_{\phi}(B) := \sup \left \{ \int_B \omega_u^d : u \in PSH(X,\omega), \ \phi-1\leq u\leq \phi  \right \}. 
$$
 Set, for $k \in \mathbb{N}$, $\mu_k:= c_k \min(f,k) \nu$ {where $c_k>0$ is chosen so that $\mu_k(X)=\int_X \omega_{\phi}^d$; this is needed in order to solve the Monge-Amp\`ere equation in the class $\Ec^1(X,\omega,\phi)$.} For $k$ large enough, $1\leq c_k\leq 2$ and $c_k\to 1$ as $k\to +\infty$. It follows from \cite[Theorem 4.25]{DDL} that there exists $u_j\in \mathcal{E}^1(X,\omega,\phi)$, $\sup_X u_j=0$, such that $\omega_{u_j}^d = \mu_j$; { by \cite[Theorem 3.12] {DDL}, $u_j\leq \phi$}.  A subsequence of $\{u_j\}$ which, by abuse of notation, will be denoted by $\{u_j\}$,  converges in $L^1(X,\mu)$ to $u\in PSH(X,\omega)$ with $u\leq \phi$. Define $v_k:= (\sup_{j\geq k} u_j)^*$. Then $v_k \searrow u$ and  $\sup_X v_k =0$. It follows from \eqref{eq: F mu proper} and \cite[Theorem 4.10]{DDL} that
\begin{flalign*}
	|{\bf E}_{\phi}(u_j)| &  \leq  \int_{X} |u_j-\phi| \omega_{u_j}^d \leq 2  \int_X |u_j-\phi| d\mu \\
	& \leq 2C (1+ |{\bf E}_{\phi}(u_j)|^{1/2}). 
\end{flalign*}
Therefore $\{|{\bf E}_{\phi}(u_j)|\}$ is bounded, hence so is $\{|{\bf E}_{\phi}(v_j)|\}$ since ${\bf E}_{\phi}$ is non-decreasing. It then follows from \cite[Lemma 4.15]{DDL} that  $u\in \mathcal{E}^1(X,\omega,\phi)$. 

Now, repeating the arguments of \cite[Theorem 4.28]{DDL}  we can show that $\omega_u^d = \mu$, finishing the proof of $(1)\Longrightarrow (2)$. 

We next prove $(2) \Longrightarrow (3)$. Assume that $\mu =\omega_u^d$ for some $u\in \Ec^1(X,\omega,\phi)$. For all $v\in \Ec^1(X,\omega,\phi)$, by \cite[Theorem 4.10]{DDL} and Proposition \ref{prop: E1 characterization} we have   
\begin{flalign*}
L_{\mu}(v)  &= \int_X (v-\phi) \omega_u^d \\
& = \int_X (v-u) \omega_u^d + \int_X (u-\phi) \omega_u^d \\
& \geq   {\bf E}_{\phi}(v) -{\bf E}_{\phi}(u)	+\int_X (u-\phi) \omega_u^d>-\infty.
\end{flalign*}
Hence $L_{\mu}$ is finite on $\Ec^1(X,\omega,\phi)$. Now, for all $v\in \Ec^1(X,\omega,\phi)$, by \cite[Theorem 4.10]{DDL}  we have 
\begin{flalign*}
	F_{\mu}(v) -F_{\mu} (u)  &= {\bf E}_{\phi}(v) -{\bf E}_{\phi}(u) -\int_X (v-u) \omega_u^d\leq 0.
\end{flalign*}
This gives $(3)$. Finally,  $(3) \Longrightarrow (1)$ is obvious. 
 \end{proof}

 \subsection{Monge-Amp\`ere equations on $\mathbb{C}^d$ with prescribed growth}
 As in the introduction we let $P$ be a convex body contained  in $(\RR^+)^d$ and fix $r>0$ such that $P\subset r \Sigma$. We assume (\ref{sigmainkp}); i.e., $\Sigma \subset kP$ for some $k\in \ZZ^+$. This ensures that $H_P$ in (\ref{logind}) is locally bounded on $\mathbb{C}^d$ (and of course $H_P\in L_P^+(\CC^d)$). Let $u\in L_P(\mathbb{C}^d)$ and define 
\begin{equation}\label{eq: tilde u}
\tilde{u}(z):= u(z) - \frac{r}{2}\log (1+ |z|^2), z\in \CC^d.	
\end{equation}
Consider the projective space $\mathbb{P}^d$ equipped with the K\"ahler metric $\omega:=r\omega_{FS}$, where 
$$
\omega_{FS}= dd^c \frac{1}{2}\log (1+ |z|^2)
$$ 
on $\mathbb{C}^d$. Then $\tilde{u}$ is bounded from above on $\mathbb{C}^d$. It thus can be extended to $\mathbb{P}^d$ as a function in $PSH(\PP^d,\omega)$. 

For a plurisubharmonic function $u$ on $\mathbb{C}^d$, we let $(dd^c u)^d$ denotes its nonpluripolar Monge-Amp\`ere measure; i.e., $(dd^c u)^d$ is the increasing limit of the sequence of measures ${\bf 1}_{\{u>-k\}} (dd^c \max(u,-k))^d$. Then
$$
\omega_{\tilde u}^d=(\omega+dd^c \tilde{u})^d = (dd^cu)^d  \ \text{on} \ \CC^d.
$$
 If $u \in L_P(\mathbb{C}^d)$ then 
$$
\int_{\CC^d} (dd^c u) ^d \leq \int_{\CC^d} (dd^c H_P)^d =d! Vol(P)=: \gamma_d=\gamma_d(P) 
$$ 
(cf., equation (2.4) in \cite{BBL}). We define 
$$
\Ec_P(\CC^d) :=\left \{ u \in L_P(\mathbb{C}^d) : \int_{\CC^d} (dd^c u)^d = \gamma_d \right \}.
$$
By the construction in \eqref{eq: tilde u} we have that $\tilde{H}_P \in PSH(\PP^d, \omega)$. We define 
$$
\tilde{\Phi}_P := P_{\omega}[\tilde{H}_P].  
$$
{The key point here, which follows from \cite[Theorem 7.2]{DDLnew}, is that $\tilde{H}_P$ has model type singularity (recall Definition \ref{modelstuff}) and hence the same singularity as $\tilde{\Phi}_P$.
Defining $\Phi_P$  on $\CC^d$ using \eqref{eq: tilde u}; i.e., for $z\in \CC^d$,
$$\Phi_P(z)=\tilde{\Phi}_P(z)+\frac{r}{2}\log (1+ |z|^2),$$
we thus have $\Phi_P  \in L_{P,+}(\CC^d)$. The advantage of using $\Phi_P$ is that, by (\ref{thm3.8}), $(dd^c \Phi_P)^d \leq \omega^d$ on $\CC^d$.} Note that $L_{P,+}(\CC^d)\subset \Ec_P(\CC^d)$. For $u,v\in L_P^+(\CC^d)$ we define
\begin{equation}\label{evu}
{E_v(u):= \frac{1}{(d+1)} \sum_{j=0}^d \int_{\CC^d} (u-v)(dd^c u)^j \wedge (dd^c v)^{d-j}.}
\end{equation}
The corresponding global energy (see (\ref{MAEN})) is defined as 
$$
{ {\bf E}_{\tilde{v}}(\tilde{u}) : = \frac{1}{(d+1)} \sum_{j=0}^d \int_{\PP^d} (\tilde{u}-\tilde{v})(\omega+dd^c \tilde{u})^j \wedge (\omega+dd^c \tilde{v})^{d-j}.}
$$
Then $E_{v}$ is non-decreasing and concave along affine curves in $L_{P,+}(\CC^d)$. We extend $E_v$ to $L_P(\mathbb{C}^d)$ in an obvious way. Note that $E_v$ may take the value $-\infty$. We define
$$
\Ec^1_P (\CC^d) := \{ u \in L_P(\mathbb{C}^d) : E_{H_P}(u) >-\infty \}.
$$
We observe that in the above definition we can replace $E_{H_P}$ by $E_{\Phi_P}$, since  for $u\in L_{P,+}(\CC^d)$, {by the cocycle property (cf. Proposition 3.3 \cite{BBL}),}
$$
E_{H_P}(u) - E_{H_P}(\Phi_P) = E_{\Phi_P}(u).
$$
We thus have the following important identification (see (\ref{E1eqn})):
\begin{equation}\label{corresp}
u \in \Ec_P^1(\CC^d) \Longleftrightarrow \tilde{u} \in \Ec^1(\PP^d, \omega, \tilde{\Phi}_P).
\end{equation}
We then have the following local version of Proposition \ref{prop: E1 characterization}: 
\begin{proposition}
	\label{pro: E1 characterization local}
	Let $u\in L_P(\mathbb{C}^d)$. Then $u\in \Ec^1_P(\CC^d)$ iff $u\in \Ec_P(\CC^d)$ and $\int_{\CC^d} (u-H_P) (dd^c u)^d>-\infty$. In particular, if supp$(dd^c u)^d$ is compact, $u \in \Ec_P^1(\CC^d)$ iff $\int_{\CC^d} (dd^c u)^d =\gamma_d$ and $\int_{\CC^d} u(dd^c u)^d>-\infty$.
\end{proposition}

{ \begin{proof} Since $\tilde{H}_P \simeq \tilde{\Phi}_P$, 
$$\int_{\PP^d} (\tilde u-\tilde{H}_P) \omega_{\tilde u}^d >-\infty \ \hbox{iff} \ \int_{\PP^d} (\tilde u-\tilde{\Phi}_P) \omega_{\tilde u}^d >-\infty$$
where $\tilde u\in PSH(\PP^d,\omega)$ and $u$ are related by (\ref{eq: tilde u}). Moreover, $\Phi_P  \in L_{P,+}(\CC^d)$ implies $u\leq \Phi_P+c$ so that $\tilde u\in PSH(\PP^d,\omega,\tilde{\Phi}_P)$. But 
$$\int_{\PP^d} (\tilde u-\tilde{H}_P) \omega_{\tilde u}^d = \int_{\CC^d} (u-H_P) (dd^c u)^d$$
and the result follows from (\ref{corresp}) by applying Proposition \ref{prop: E1 characterization} to $\tilde u$.
For the last statement, note that for general $u\in L_P(\mathbb{C}^d)$ we may have $\int_{\CC^d} H_P (dd^c u)^d=+\infty$, but if $(dd^c u)^d$ has compact support then $\int_{\CC^d} H_P(dd^c u)^d$ is finite.  
\end{proof}}

Note that Theorem \ref{thm: mass monotone} and Proposition \ref{prop: mixed MA mass global} give the following result:
\begin{theorem}
	\label{thm: mass monotone local}
	Let $u_1,...,u_d$ be functions in $\Ec_{P}(\CC^d)$. Then 
	$$
	\int_{\CC^d} dd^c u_1 \wedge ... \wedge dd^c u_d = \gamma_d. 
	$$
\end{theorem}

For $u_1,...,u_n\in L_{P,+}(\CC^d)$ Theorem \ref{thm: mass monotone local} was proved in  \cite[Proposition 2.7]{Bay}.

Having the correspondence (\ref{corresp}) we can state a local version of Theorem \ref{thm: existence in E1}; this will be used in the sequel. Let $\mathcal M_P(\CC^d)$ denote the set of all positive Borel measures $\mu$ on $\CC^d$ with $\mu(\CC^d)=d! Vol(P)=\gamma_d$.

\begin{theorem}\label{thm: existence in E1 local}
	Assume that $\mu\in \mathcal M_P(\CC^d)$ is a positive nonpluripolar Borel measure. The following are equivalent
	\begin{enumerate}
		\item $\Ec^1_P(\CC^d) \subset L^1(\CC^d,\mu)$; 
		\item there exists $u\in \Ec^1_P(\CC^d)$ such that $(dd^c u)^d =\mu$;
		\item there exists $u \in \Ec^1_P(\CC^d)$ such that $${\mathcal F}_{\mu}(u) = \max_{v\in \Ec^1_P(\CC^d)} {\mathcal F}_{\mu}(v)<+\infty.$$
	\end{enumerate}
\end{theorem}
A priori the functional ${\mathcal F}_{\mu}$ is defined for $u\in \Ec^1_P(\CC^d)$ by 
$$
{\mathcal F}_{\mu,\Phi_P}(u) := E_{\Phi_P}(u) -\int_{\CC^d} (u-\Phi_P) d\mu.
$$
However, using this notation, since 
$$
{\mathcal F}_{\mu,\Phi_P}(u) - {\mathcal F}_{\mu,H_P}(u) = {\mathcal F}_{\mu,\Phi_P}(H_P),
$$
in  statement $(3)$ of Theorem \ref{thm: existence in E1 local} we can take either of the two definitions ${\mathcal F}_{\mu,\Phi_P}$ or ${\mathcal F}_{\mu,H_P}$ for ${\mathcal F}_{\mu}$.

\begin{remark} \label{rem: compact support}
	If $\mu$ has compact support in $\CC^d$ then $\int_{\CC^d} \Phi_P d\mu$ and $\int_{\CC^d} H_P d\mu$ are finite. Therefore, the functional ${\mathcal F}_{\mu}$ can be replaced by 
	$$
	u \mapsto E_{H_P}(u) - \int_{\CC^d} u d\mu.
	$$ 
\end{remark}

Using the remark, for $\mu\in \mathcal M_P(\CC^d)$ with compact support, it is natural to define the Legendre-type transform of $E_{H_P}$:
\begin{equation}\label{estar2} E^*(\mu):= \sup_{u\in \Ec_P^1(\CC^d)}[E_{H_P}(u) -\int_{\CC^d} u d\mu].
\end{equation}
This functional, which will appear in the rate function for our LDP, will be given a more concrete interpretation using $P-$pluripotential theory in section 4; cf., equation (\ref{eandj}).

Finally, for future use, we record the following consequence of Lemma \ref{lem: pluripolar} and the correspondence (\ref{corresp}).

\begin{lemma}
	\label{lem: pluripolar local}
	If $E\subset \CC^d$ is pluripolar then there exists $u\in \Ec^1_P(\CC^d)$ such that $E\subset \{u=-\infty\}$.
\end{lemma}

\section{$P-$pluripotential theory notions}

Given $E\subset \CC^d$, the $P-$extremal function of $E$ is 
$$V^*_{P,E}(z):=\limsup_{\zeta \to z}V_{P,E}(\zeta)$$ where
$$V_{P,E}(z):=\sup \{u(z):u\in L_P(\CC^d), \ u\leq 0 \ \hbox{on} \ E\}.$$
For $K\subset \CC^d$ compact, $w:K\to \RR^+$ is an admissible weight function on $K$ if $w\geq 0$ is an uppersemicontinuous function with $\{z\in K:w(z)>0\}$ nonpluripolar. Setting $Q:= -\log w$, we write $Q\in \mathcal A(K)$ and define the {\it weighted $P-$extremal function} $$V^*_{P,K,Q}(z):=\limsup_{\zeta \to z}V_{P,K,Q}(\zeta)$$ where
$$V_{P,K,Q}(z):=\sup \{u(z):u\in L_P(\CC^d), \ u\leq Q \ \hbox{on} \ K\}. $$
If $Q=0$ we write $V_{P,K,Q}=V_{P,K}$, consistent with the previous notation. For $P=\Sigma$, 
$$V_{\Sigma,K,Q}(z)=V_{K,Q}(z):= \sup \{u(z):u\in L(\CC^d), \ u\leq Q \ \hbox{on} \ K\} $$
is the usual weighed extremal function as in Appendix B of \cite{ST}.

We write (omitting the dependence on $P$)
$$\mu_{K,Q}:= (dd^c V^*_{P,K,Q})^d \ \hbox{and} \ \mu_{K}:= (dd^c V^*_{P,K})^d$$
for the Monge-Amp\`ere measures of $V^*_{P,K,Q}$ and $V^*_{P,K}$ (the latter if $K$ is not pluripolar). Proposition 2.5 of \cite{BBL} states that
$$supp(\mu_{K,Q})\subset \{z\in K: V_{P,K,Q}^*(z)\geq Q(z)\}$$
and $V_{P,K,Q}^*=Q$ q.e. on $supp(\mu_{K,Q})$, i.e., off of a pluripolar set.

\subsection{Energy}

We  recall some results and definitions from \cite{BBL}. For $u,v \in L_{P,+}(\CC^d)$, we define the {\it mutual energy} 
$$
\mathcal E (u,v):= \int_{\CC^d} (u-v)\sum_{j=0}^d (dd^cu)^j\wedge (dd^cv)^{d-j}.
$$
For simplicity, when $v=H_P$, we denote the associated (normalized) energy functional by $E$:
 $$E(u):=E_{H_P}(u)= \frac{1}{d+1}\sum_{j=0}^d\int_{\CC^d} (u-H_P) dd^cu^j\wedge (dd^cH_P)^{d-j}$$
 (recall (\ref{evu})).

For $u,u',v\in L_{P,+}(\CC^d)$, and for $0\leq t\leq 1$, we define
$$f(t):= {\mathcal E} (u+t(u'-u),v),$$
From Proposition 3.1 in \cite{BBL}, $f'(t)$ exists for $0\leq t \leq 1$ and 
$$ 
f'(t)= (d+1)\int_{\CC^d} (u'-u) (dd^c(u+t(u'-u)))^d
$$
Hence, taking $v=H_P$, we have, for $F(t):=E( u+t(u'-u))$, that 
$$F'(t)=\int_{\CC^d} (u'-u) (dd^c(u+t(u'-u)))^d.$$
Thus $F'(0)=\int_{\CC^d} (u'-u) (dd^c u)^d$ and we write
\begin{equation}
\label{diff02}
<E' (u), u'-u>:=\int (u'-u)(dd^cu)^d.
\end{equation}

We need some applications of a global domination principle. The following version, sufficient for our purposes, follows from \cite{DDL}, Corollary 3.10 (see also Corollary A.2 of \cite{PE}).

\begin{proposition} \label{gctp} Let $u \in L_P(\CC^d)$ and $v \in \mathcal E_P(\CC^d)$ with $u \leq v$ a.e. $(dd^cv)^d$. Then $u \leq v$ in $\CC^d$.
\end{proposition}

This will be used to prove an approximation result, Proposition \ref{goodapprox}, which itself will be essential in the sequel. First we need a lemma.

\begin{lemma}
	\label{lem: energy estimate}
	Assume that $\varphi \leq u,v\leq H_P$ are functions in $\Ec^1_P(\CC^d)$. Then for all $t>0$,
	$$
	\int_{\{u\leq H_P -2t\}} (H_P-u) (dd^c v)^d \leq 2^{d+1} \int_{\{\varphi\leq H_P -t\}} (H_P-\varphi) (dd^c \varphi)^d. 
	$$
	In particular, the left hand side converges to $0$ as $t\to +\infty$ uniformly in $u,v$. 
\end{lemma}
\begin{proof}
	For $s>0$, we have the following inclusions of sets:
	$$
	(u\leq H_P-2s)  \subset \left (\varphi \leq \frac{v+H_P}{2} -s\right )\subset (\varphi \leq H_P-s).
	$$
	We first note that the left hand side in the lemma is equal to
	\begin{equation}
		\label{basic integral}
		\int_{\{u\leq H_P -2t\}} (H_P-u) (dd^c v)^d
	\end{equation}
	$$=2t\int_{\{u\leq H_P -2t\}} (dd^c v)^d+\int_{2t}^{\infty}\left(\int_{\{u\leq H_P -s\}} (dd^c v)^d\right)ds.$$
	We claim that, for all $s>0$, 
	\begin{equation}
		\label{comparison principle}
		\int_{\{u\leq H_P -2s\}} (dd^c v)^d\leq 2^d\int_{\{\varphi\leq H_P -s\}} (dd^c \varphi)^d.
	\end{equation}
	Indeed, the comparison principle (\cite[Corollary 3.6]{DDL}) and the inclusions of sets above give
	$$\int_{\{u\leq H_P -2s\}} (dd^c v)^d\leq \int_{\{\varphi \leq \frac{v+H_P}{2} -s\}} (dd^c v)^d\leq 2^d\int_{\{\varphi \leq \frac{v+H_P}{2} -s\}} \left(dd^c \frac{v+H_P}{2}\right)^d$$
	$$\leq 2^d\int_{\{\varphi \leq \frac{v+H_P}{2} -s\}} (dd^c \varphi)^d\leq 2^d\int_{\{\varphi \leq H_P-s\}} (dd^c \varphi)^d.$$
	The claim is proved. Using \eqref{comparison principle} and \eqref{basic integral} we obtain 
$$\int_{\{u\leq H_P -2t\}} (H_P-u) (dd^c v)^d$$
$$\leq 2^{d+1} t \int_{\{\varphi\leq H_P-t\}}(dd^c \varphi)^d+ 2^{d+1} \int_{t}^{+\infty} \left( \int_{\{\varphi \leq H_P-s\}}(dd^c \varphi)^d\right)ds$$
$$=2^{d+1} \int_{\{\varphi\leq H_P -t\}} (H_P-\varphi) (dd^c \varphi)^d.$$
\end{proof}

\begin{proposition} \label{goodapprox} Let $u\in \mathcal E_P^1(\CC^d)$ with $(dd^c u)^d=\mu$ having support in a nonpluripolar compact set $K$ so that $\int_K ud\mu > -\infty$ from Proposition \ref{pro: E1 characterization local}. Let $\{Q_j\}$ be a sequence of continuous functions on $K$ decreasing to $u$ on $K$. Then $u_j:= V_{P,K,Q_j}^* \downarrow  u$ on $\CC^d$ and 
$\mu_j:=(dd^cu_j)^d$ is supported in $K$. In particular, $\mu_j \to \mu =(dd^cu)^d$ weak-*. Moreover,
\begin{equation}\label{weak*}\lim_{j\to \infty} \int_{K} Q_jd\mu_j = \lim_{j\to \infty} \int_{K} Q_jd\mu=\int_K ud\mu>-\infty .\end{equation}
\end{proposition}

\begin{proof} 
We can assume $\{Q_j\}$ are defined and decreasing to $u$ on the closure of a bounded open neighborhood $\Omega$ of $K$. By adding a negative constant we can assume that $Q_1\leq 0$ on $\Omega$. Since $\{Q_j\}$ is decreasing, so is the sequence $\{u_j\}$. {Moreover, by \cite[Proposition 5.1]{BT82}  $u_j \leq Q_j$ on $K\setminus E_j$ where $E_j$ is pluripolar. But $u$ is a competitor in the definition of $V_{P,K,Q_j}$ so that $u \leq u_j$ on $\CC^d$. Thus $\tilde u:= \lim_{j\to \infty} u_j  \geq u$ everywhere and $\tilde u \leq  u$ on $K\setminus E$, where $E:= \cup_{j} E_j$ is a pluripolar set.  Since $(dd^c u)^d$ put no mass on pluripolar sets,  
$$
\int_{\{u<\tilde{u}\}}(dd^c u)^d\leq  \int_{E \cup(\CC^d\setminus K)}(dd^c u)^d  =0.
$$
 It thus follows from Proposition \ref{gctp} that $\tilde u \leq  u$, hence $\tilde{u}=u$ on $\CC^d$.} 
 
 The second equality in \eqref{weak*} follows from the monotone convergence theorem. It remains to prove that 
$$ \lim_{j\to \infty}	\int_K (-Q_j) d\mu_j = \int_K (-u) d\mu.$$

 For each $k$ fixed and $j\geq k$ we have 
 $$
 \int_K (-Q_j) d\mu_j \geq \int_K (-Q_k) d\mu_j= \int_{\Omega} (-Q_k) d\mu_j,
 $$ 
 hence $\liminf_{j\to \infty} \int_K (-Q_j) d\mu_j \geq \int_K (-Q_k) d\mu$ since $\Omega$ is open and $\mu_j,\mu$ are supported on $K$.  Letting $k\to +\infty$ we arrive at 
 $$
\liminf_{j\to \infty} \int_K (-Q_j) d\mu_j \geq \int_K (-u) d\mu.
 $$
 
 It remains to prove that
 $$
\limsup_{j\to \infty} \int_K (-Q_j) d\mu_j \leq \int_K (-u) d\mu.
 $$
The sequence $\{u_j\}$ is not necessarily uniformly bounded below on $K$. However, using the facts that
$Q_j \geq u$ and $H_P$ is continuous in $\CC^d$, it suffices to prove that 
 \begin{equation}
 	\label{eq: weak convergence Qj}
 	\limsup_{j\to \infty}\int_K (H_P-u) (dd^c u_j)^d \leq \int_K (H_P-u) (dd^c u)^d. 
 \end{equation} 
 To verify (\ref{eq: weak convergence Qj}), we use Lemma \ref{lem: energy estimate}.
 
 By adding a negative constant we can assume that $u_j\leq H_P$.  For a function $v$ and for $t>0$ we define $v^t:= \max(v,H_P-t)$. Note that for each $t$ the sequence $\{u^t_j\}$ is locally uniformly bounded below. Define
 $$a(t):=2^{d+1} \int_{\{u\leq H_P-t/2\}} (H_P-u)(dd^c u)^d.$$ 
 Since $u\in \Ec_P^1(\CC^d)$, from Proposition \ref{pro: E1 characterization local} we have $a(t) \to 0$ as $t\to +\infty$. By Lemma \ref{lem: energy estimate} we have
 \begin{equation}\label{fromenest}
\sup_{j\geq 1}\int_{\{u\leq H_P-t\}} (H_P-u) (dd^c u_j)^d\leq  a(t). 
  \end{equation}
By the plurifine property of non-pluripolar Monge-Amp\`ere measures  \cite[Proposition 1.4]{[BEGZ]} and (\ref{fromenest}) we  have
\begin{flalign*}
	\int_{K} (H_P-u) (dd^c u_j)^d & \leq \int_{K \cap \{u>H_P-t\}} (H_P-u) (dd^c u_j)^d +a(t)\\
	& =\int_{K \cap \{u>H_P-t\}} (H_P-u^t) (dd^c u_j^t)^d +a(t)\\
	&\leq \int_{K} (H_P-u^t) (dd^c u_j^t)^d +a(t).
\end{flalign*}
Since $H_P$ is bounded in $\Omega$, it follows from \cite[Theorem 4.26]{GZ} that the sequence of positive Radon measures $(H_P-u^t) (dd^c u_j^t)^d$ converges weakly on $\Omega$ to $(H_P-u^t)(dd^c u^t)^d$. Since $K$ is compact it then follows that 
\begin{flalign*}
	\limsup_{j}\int_{K} (H_P-u) (dd^c u_j)^d  & \leq \int_K (H_P-u^t) (dd^c u^t)^d +a(t).
\end{flalign*}
We finally let $t\to +\infty$ to conclude the proof in the following manner: 
\begin{flalign*}
	\int_K (H_P-u^t) (dd^c u^t)^d &  \leq \int_{K\cap \{u>H_P-t\}} (H_P-u^t)(dd^c u^t)^d+a(t)\\
	& \leq \int_{K} (H_P-u)(dd^c u)^d +a(t),
\end{flalign*}
where in the first estimate we have used $\{u\leq H_P-t\}=\{u^t\leq H_P-t\}$ and Lemma \ref{lem: energy estimate} and in the last estimate we use again the plurifine property. 
\end{proof}

We now give an alternate description of the Legendre-type transform $E^*$ from (\ref{estar2}) which 
will be related to the the rate function in a large deviation principle. Given $K\subset \CC^d$ compact, we let ${\mathcal M}_P(K)$ denote the space of positive measures on $K$ of total mass $\gamma_d$ and we let $C(K)$ denote the set of continuous, real-valued functions on $K$.

\begin{proposition} \label{gap} Let $K$ be a nonpluripolar  compact set and $\mu  \in {\mathcal M}_P(K)$. Then
$$E^*(\mu)=\sup_{v\in C(K)} [E(V^*_{P,K,v})-\int_K v d\mu].$$
\end{proposition}

\begin{proof} 
We first treat the case when $E^*(\mu)=+\infty$. By Theorem \ref{thm: existence in E1 local} there exists $u\in \Ec_P^1(\CC^d)$ such that  $\int_K u d\mu =-\infty$. We take a decreasing sequence  $Q_j \in C(K)$ such that $Q_j\downarrow u$ on $K$ and set $u_j:= V^*_{P,K,Q_j}$.  Then $\{u_j\}$ are decreasing; since $u\in \Ec_P^1(\CC^d)$ and $E$ is non-decreasing, $\{E(u_j)\}$ is uniformly bounded and we obtain 
$$
E(V^*_{P,K,Q_j}) -\int_K Q_j d\mu \to +\infty,
$$
proving the proposition in this case.

Assume now that $E^*(\mu)<+\infty$. Theorem \ref{thm: existence in E1 local} ensures that $\int_{\CC^d} u d\mu >-\infty$ for all $u \in \Ec^1_P(\CC^d)$. By Lemma \ref{lem: pluripolar local},  $\mu$ puts no mass on pluripolar sets. 
From monotonicity of $E$ and the definition of $E^*$ in (\ref{estar2}) we have
$$E^*(\mu)\geq \sup_{v\in C(K)}[E(V^*_{P,K,v})-\int_K v d\mu].$$
Here we have used that 
$$V^*_{P,K,v}\leq v \ \hbox{q.e. on} \ K \   \hbox{for} \ v\in C(K).$$ 
For the reverse inequality, fix $u\in \Ec_P^1(\CC^d)$.  Let $\{Q_j\}$ be a sequence of continuous functions on $K$ decreasing to $u$ on $K$ and set $u_j:= V^*_{P,K,Q_j}$.  Given $\epsilon >0$, we can choose $j$ sufficiently large so that, by monotone convergence,
$$\int_K Q_j d\mu \leq  \int_K u d\mu + \epsilon;$$
and, by monotonicity of $E$, 
$$E(V^*_{P,K,Q_j})\geq E(u).$$
Hence
$$E(V^*_{P,K,Q_j}) - \int_K Q_j d\mu \geq E(u)- \int_K u d\mu -\epsilon$$
so that 
$$\sup_{v\in C(K)}[E(V^*_{P,K,v})-\int_K v d\mu]\geq E^*(\mu)$$
and equality holds.

\end{proof}

 \subsection{Transfinite diameter.}

Let $d_n=d_n(P)$ denote the dimension of the vector space $Poly(nP)$. We write
$$Poly(nP)= \hbox{span} \{e_1,...,e_{d_n}\}$$ 
where $\{e_j(z):=z^{\alpha(j)}\}_{j=1,...,d_n}$ are the standard basis monomials. Given $\zeta_1,...,\zeta_{d_n}\in \CC^d$, let
\begin{equation} \label{vdm}VDM(\zeta_1,...,\zeta_{d_n}):=\det [e_i(\zeta_j)]_{i,j=1,...,d_n}  \end{equation}
$$= \det
\left[
\begin{array}{ccccc}
 e_1(\zeta_1) &e_1(\zeta_2) &\ldots  &e_1(\zeta_{d_n})\\
  \vdots  & \vdots & \ddots  & \vdots \\
e_{d_n}(\zeta_1) &e_{d_n}(\zeta_2) &\ldots  &e_{d_n}(\zeta_{d_n})
\end{array}
\right]$$
and for $K\subset \CC^d$ compact let
$$V_n =V_n(K):=\max_{\zeta_1,...,\zeta_{d_n}\in K}|VDM(\zeta_1,...,\zeta_{d_n})|.$$
It was shown in \cite{BBL} that
\begin{equation} \label{tdlim} \delta(K):=\delta(K,P):= \lim_{n\to \infty}V_{n}^{1/l_n} \end{equation} exists where 
$$l_n:=\sum_{j=1}^{d_n} {\rm deg}(e_j)= \sum_{j=1}^{d_n} |\alpha(j)|$$ is the sum of the degrees of the basis monomials for $ Poly(nP)$. We call $\delta(K)$ the {\it $P-$transfinite diameter} of $K$. More generally, for $w$ an admissible weight function on
$K$ and $\zeta_1,...,\zeta_{d_n}\in K$, let
\begin{equation}\label{vdmq} VDM_n^Q(\zeta_1,...,\zeta_{d_n}):=VDM(\zeta_1,...,\zeta_{d_n})w(\zeta_1)^{n}\cdots w(\zeta_{d_n})^{n}\end{equation}
$$= \det
\left[
\begin{array}{ccccc}
 e_1(\zeta_1) &e_1(\zeta_2) &\ldots  &e_1(\zeta_{d_n})\\
  \vdots  & \vdots & \ddots  & \vdots \\
e_{d_n}(\zeta_1) &e_{d_n}(\zeta_2) &\ldots  &e_{d_n}(\zeta_{d_n})
\end{array}
\right]\cdot w(\zeta_1)^{n}\cdots w(\zeta_{d_n})^{n}$$
be a {\it weighted Vandermonde determinant}. Let
$$W_n(K):=\max_{\zeta_1,...,\zeta_{d_n}\in K}|VDM_n^Q(\zeta_1,...,\zeta_{d_n})|.
$$
An {\it $n-$th weighted $P-$Fekete set for $K$ and $w$} is a set of $d_n$ points $\zeta_1,...,\zeta_ {d_n}\in K$ with the property that
$$|VDM_n^Q(\zeta_1,...,\zeta_{d_n})|=W_n(K).$$

The limit
$$ \delta^Q(K):=\delta^Q(K,P):=\lim_{n\to \infty}W_{n}(K)^{1/l_n} $$
exists and is called  the {\it weighted $P-$transfinite diameter}. The following was proved in \cite{BBL}.
 
 \begin{theorem} \label{asympwtdfek} {\bf [Asymptotic Weighted $P-$Fekete Measures]} Let $K\subset \CC^d$ be compact with admissible weight $w$. For each $n$, take points $z_1^{(n)},z_2^{(n)},\cdots,z_{d_n}^{(n)}\in K$ for which 
\begin{equation}\label{wam}
 \lim_{n\to \infty}\bigl[|VDM_n^Q(z_1^{(n)},\cdots,z_{d_n}^{(n)})|\bigr]^{1 \over  l_n}=\delta^Q(K)
\end{equation}
({\it asymptotically} weighted $P-$Fekete arrays) and let $\mu_n:= \frac{1}{d_n}\sum_{j=1}^{d_n} \delta_{z_j^{(n)}}$. Then
$$
\mu_n \to \frac{1}{\gamma_d}\mu_{K,Q} \ \hbox{weak}-*.
$$
\end{theorem}

 Another ingredient we will use is a Rumely-type relation between transfinite diameter and energy of $V_{P,K,Q}^*$ from \cite{BBL}.

\begin{theorem} \label{energyrumely}
Let $K\subset \CC^d$ be compact and $w=e^{-Q}$ with $Q\in C(K)$. Then
\begin{equation} \label{enrum}
\log \delta^Q(K)=  \frac{-1}{\gamma_ddA}\mathcal E(V_{P,K,Q}^*,H_P)=  \frac{-(d+1)}{\gamma_ddA}E(V_{P,K,Q}^*).
\end{equation}
\end{theorem}

\noindent Here $A=A(P,d)$ was defined in \cite{BBL}; we recall the definition. For $P=\Sigma$ so that $Poly(n\Sigma)=\mathcal P_n$, we have 
$$d_n(\Sigma)={d+n \choose d}=0(n^d/d!) \ \hbox{and} \ l_n(\Sigma)= \frac{d}{d+1}nd_n(\Sigma).$$
For a convex body $P\subset (\RR^+)^d$, define $f_n(d)$ by writing
$$
l_n = f_n(d)  \frac{nd}{d+1}d_n = f_n(d) \frac{l_n (\Sigma)}{d_n(\Sigma)}d_n.
$$
Then the ratio $l_n/d_n$ divided by $l_n (\Sigma)/d_n(\Sigma)$ has a limit; i.e., 
\begin{equation}\label{key}\lim_{n\to \infty} f_n(d)=: A= A(P,d).\end{equation}

\subsection{Bernstein-Markov}

For $K\subset \CC^d$ compact, $w=e^{- Q}$ an admissible weight function on $K$, and $\nu$ a finite measure on $K$, we say that the triple $(K,\nu,Q)$ satisfies a weighted Bernstein-Markov property if for all $p_n\in \mathcal P_n$, 
\begin{equation}\label{wtdbm}||w^np_n||_K \leq M_n ||w^np_n||_{L^2(\nu)} \ \hbox{with} \ \limsup_{n\to \infty} M_n^{1/n} =1.\end{equation}
Here, $||w^np_n||_K:=\sup_{z\in K} |w(z)^np_n(z)|$ and 
$$ ||w^np_n||_{L^2(\nu)}^2:=\int_K |p_n(z)|^2  w(z)^{2n} d\nu(z).$$ 
Following \cite{Bay}, given $P\subset (\RR^+)^d$ a convex body, we say that a finite measure $\nu$ with support in a compact set $K$ is a Bernstein-Markov measure for the triple $(P,K,Q)$ if (\ref{wtdbm}) holds for all $p_n\in Poly(nP)$. 

For any $P$ there exists $A=A(P)>0$ with $Poly(nP) \subset \mathcal P_{An}$ for all $n$. Thus if $(K,\nu,Q)$ satisfies a weighted Bernstein-Markov property, then $\nu$ is a Bernstein-Markov measure for $(P,K,\tilde Q)$ where $\tilde Q = AQ$. In particular, if $\nu$ is a {\it strong Bernstein-Markov measure} for $K$; i.e., if $\nu$ is a weighted Bernstein-Markov measure for any $Q\in C(K)$, then for any such $Q$, $\nu$ is a Bernstein-Markov measure for the triple $(P,K,Q)$. Strong Bernstein-Markov measures exist for any nonpluripolar compact set; cf., Corollary 3.8 of \cite{PELD}. The paragraph following this corollary gives a sufficient mass-density type condition for a measure to be a strong Bernstein-Markov measure.
 
 Given $P$, for $\nu$ a finite measure on $K$ and $Q\in \mathcal A(K)$, define 
\begin{equation}\label{wtdzn}Z_n:=Z_n(P,K,Q,\nu):= \int_K \cdots \int_K |VDM_n^Q(z_1,...,z_{d_n}
)|^2 d\nu(z_1) \cdots d\nu(z_{d_n}).\end{equation}
The main consequence of using a Bernstein-Markov measure for $(P,K,Q)$ is the following:

\begin{proposition}
\label{weightedtd}
Let $K\subset \CC^d$ be a compact set and let $Q\in \mathcal A(K)$. If $\nu$ is a Bernstein-Markov measure for $(P,K,Q)$ then 
\begin{equation}
\label{zeen}\lim_{k\to \infty} Z_n^{\frac{1}{2l_n}
} =  \delta^Q(K).\end{equation}\end{proposition}

\begin{proof} That $\limsup_{k\to \infty} Z_n^{\frac{1}{2l_n}
} \leq \delta^Q(K)$ is clear. Observing from (\ref{vdm}) and (\ref{vdmq}) that, fixing all variables but $z_j$, 
$$z_j\to VDM_n^Q(z_1,...,z_j,...,z_{d_n})=w(z_j)^np_n(z_j)$$ for some $p_n\in Poly(nP)$, to show $\liminf_{k\to \infty} Z_n^{\frac{1}{2l_n}
} \geq  \delta^Q(K)$ one starts with an $n-$th weighted $P-$Fekete set for $K$ and $w$ and repeatedly applies the weighted Bernstein-Markov property.
\end{proof}

Recall ${\mathcal M}_P(K)$ is the space of positive measures on $K$ with total mass $\gamma_d$. With the weak-* topology, this is a separable, complete metrizable space. A neighborhood basis 
of $\mu \in {\mathcal M}_P(K)$ can be given by sets 
\begin{equation}\label{nbhdbase}G(\mu, k, \epsilon) := \{\sigma  \in {\mathcal M}_P(K):
|\int_K (\Re z)^{\alpha}(\Im z)^{\beta} (d\mu - d\sigma )| < \epsilon\end{equation}
$$\hbox{for} \ 0\leq |\alpha|+|\beta| \leq k\}$$
where $\Re z =(\Re z_1,..., \Re z_n)$ and $\Im z =(\Im z_1,..., \Im z_n)$.

Given $\nu$ as in Proposition \ref{weightedtd}, we define a probability measure $Prob_n$ on $K^{d_n}$ via, for a Borel set $A\subset K^{d_n}$,
\begin{equation}\label{probk}Prob_n(A):=\frac{1}{Z_n}\cdot \int_A  |VDM_n^Q(z_1,...,z_{d_n
})|^2 \cdot d\nu(z_1) \cdots d\nu(z_{d_n
}).\end{equation}
We immediately obtain the following:
		
\begin{corollary} \label{largedev} Let $\nu$ be a Bernstein-Markov measure for $(P,K,Q)$. Given $\eta >0$, define
 \begin{equation}\label{aketa}A_{n,\eta}:=\{(z_1,...,z_{d_n})\in K^{d_n}: |VDM_n^Q(z_1,...,z_{d_n})|^2 \geq (\delta^Q(K) -\eta)^{2l_n}\}.\end{equation}
Then there exists $n^*=n^*(\eta)$ such that for all $n>n^*$, 
$$Prob_n(K^{d_n}\setminus A_{n,\eta})\leq \left(1-\frac{\eta}{2 \delta^Q(K)}\right)^{2l_n}.$$
	\end{corollary}	
	
{\begin{remark} Corollary \ref{largedev} was proved in \cite{PELD}, Corollary 3.2, for $\nu$ a probability measure but an obvious modification works for $\nu(K)<\infty$.

\end{remark}}
		
	Using (\ref{probk}), we get an induced probability measure ${\bf P}$ on the infinite product space of arrays $\chi:=\{X=\{x_j^{(n)}\}_{n=1,2,...; \ j=1,...,d_n}: x_j^{(n)} \in K\}$: 
	$$(\chi,{\bf P}):=\prod_{n=1}^{\infty}(K^{d_n},Prob_n).$$
	
\begin{corollary} \label{niceprob} 
Let $\nu$ be a Bernstein-Markov measure for $(P,K,Q)$. For ${\bf P}$-a.e. array $X=\{x_j^{(n)}\}\in \chi$, 
$$\nu_n :=\frac{1}{d_n}\sum_{j=1}^{d_n}\delta_{x_j^{(n)}}\to \frac{1}{\gamma_d}\mu_{K,Q} \ \hbox{weak-*}.$$
\end{corollary}  
\begin{proof} From Theorem \ref{asympwtdfek} it suffices to verify for ${\bf P}$-a.e. array $X=\{x_j^{(n)}\}$
\begin{equation}\label{borcan}\liminf_{n\to \infty} \bigl(|VDM_n^Q(x_1^{(n)},...,x_{d_n}^{(n)})|\bigr)^{\frac{1}{l_n}}  =  \delta^Q(K).\end{equation}
Given $\eta >0$, the condition that for a given array $X=\{x_j^{(n)}\}$ we have
$$\liminf_{n\to \infty} \bigl(|VDM_n^Q(x_1^{(n)},...,x_{d_n}^{(n)})|\bigr)^{\frac{1}{l_n}} \leq  \delta^Q(K)-\eta$$
means that $(x_1^{(n)},...,x_{d_n}^{(n)})\in K^{d_n}\setminus A_{n,\eta}$ for infinitely many $n$. Setting 
$$E_n:=\{X \in \chi: (x_1^{(n)},...,x_{d_n}^{(n)})\in K^{d_n}\setminus A_{n,\eta}\},$$
we have
$${\bf P}(E_n)\leq Prob_n(K^{d_n}\setminus A_{n,\eta})\leq (1-\frac{\eta}{2\delta^Q(K)})^{2l_n}$$
and $\sum_{n=1}^{\infty} {\bf P}(E_n)<+\infty$. By the Borel-Cantelli lemma, 
$${\bf P}(\limsup_{n\to \infty} E_n)={\bf P}(\bigcap_{n=1}^{\infty} \bigcup_{k\geq n}^{\infty} E_k)=0.$$ Thus, with probability one, only finitely many $E_n$ occur,  and (\ref{borcan}) follows.
\end{proof}

The main goal in the rest of the paper is to verify a stronger probabilistic result -- a large deviation principle -- and to explain this result in $P-$pluripotential-theoretic terms.

 \section{Relation between $E^*$ and $J,J^Q$ functionals.}
	\label{sec:eandp} 
	We define some functionals on $\mathcal M_P(K)$ using $L^2-$type notions which act as a replacement for an energy functional on measures. Then we show these functionals $\overline J(\mu)$ and $\underline J(\mu)$ defined using a ``$\limsup$'' and a ``$\liminf$'' coincide (see Definitions \ref{jwmu} and \ref{jwmuq}); this is the essence of our first proof of the large deviation principle, Theorem \ref{ldp}. Using Proposition \ref{gap}, we relate this functional with $E^*$ from (\ref{estar2}).

	Fix a nonpluripolar compact set $K$ and a strong Bernstein-Markov measure $\nu$ on $K$. For simplicity, we normalize so that $\nu$ is a probability measure. Recall then for any $Q\in C(K)$, $\nu$ is a Bernstein-Markov measure for the triple $(P,K,Q)$. Given $G\subset {\mathcal M}_P(K)$ open, for each $s=1,2,...$ we set
\begin{equation}\label{nbhddef}\tilde G_s:= \{{\bf a} =(a_1,...,a_s)\in K^s: \frac{\gamma_d}{s}\sum_{j=1}^s \delta_{a_j}\in G\}.\end{equation}
Define, for $n=1,2,...$, 
$$J_n(G):=[\int_{\tilde G_{d_n}}|VDM_n({\bf a})|^{2}d\nu ({\bf a})]^{1/2l_n}.$$

\begin{definition} \label{jwmu} For $\mu \in \mathcal M_P(K)$ we define
$$\overline J(\mu):=\inf_{G \ni \mu} \overline J(G) \ \hbox{where} \ \overline J(G):=\limsup_{n\to \infty} J_n(G);$$
$$\underline J(\mu):=\inf_{G \ni \mu} \underline J(G) \ \hbox{where} \ \underline J(G):=\liminf_{n\to \infty} J_n(G).$$
\end{definition}

The infima are taken over all neighborhoods $G$ of the measure $\mu$ in ${\mathcal M}_P(K)$. 
A priori, $\overline J,\underline J$ depend on $\nu$. These functionals are nonnegative but can take the value zero. Intuitively, we are taking a ``limit'' of $L^2(\nu)$ averages of discrete, equally weighted approximants $\frac{\gamma_d}{s}\sum_{j=1}^s \delta_{a_j}$ of $\mu$. An ``$L^{\infty}$'' version of $\overline J,\underline J$ was introduced in \cite{PE} where $J_n(G)$ is replaced by 
\begin{equation}\label{wfcl} W_n(G):=\sup_{{\bf a}\in \tilde G_{d_n}}|VDM_n({\bf a})|^{1/l_n}\geq J_n(G).\end{equation}

The weighted versions of these functionals are defined for $Q\in \mathcal A(K)$ using
\begin{equation}\label{jkqmu}J^Q_n(G):=[\int_{\tilde G_{d_n}}|VDM^Q_n({\bf a})|^{2}d\nu ({\bf a})]^{1/2l_n}.\end{equation}

\begin{definition} \label{jwmuq} For $\mu \in \mathcal M_P(K)$ we define
$$\overline J^Q(\mu):=\inf_{G \ni \mu} \overline J^Q(G) \ \hbox{where} \ \overline J^Q(G):=\limsup_{n\to \infty} J^Q_n(G);$$
$$\underline J^Q(\mu):=\inf_{G \ni \mu} \underline J^Q(G) \ \hbox{where} \ \underline J^Q(G):=\liminf_{n\to \infty} J^Q_n(G).$$
\end{definition}

The uppersemicontinuity of $\overline J, \overline J^Q, \underline J$ and $\underline J^Q$ on ${\mathcal M}_P(K)$ (with the weak-* topology) follows as in Lemma 3.1 of \cite{PE}.  Set 
$$b_d=b_d(P):=\frac{d+1}{Ad\gamma_d}.$$
\begin{proposition}\label{prop: relation J JQ}  Fix $Q\in C(K)$. Then 
\begin{enumerate}
\item $\overline J^Q(\mu)\leq   \delta^Q (K)$;
\item $\overline J(\mu)=\overline J^Q(\mu)\cdot (e^{\int_K Qd\mu})^{b_d} $; 
\item $\log \overline J(\mu)\leq  \inf_{v\in  C(K)} [\log  \delta^{v} (K) +b_d\int_K vd\mu]$; 
\item $\log \overline J^Q(\mu)\leq  \inf_{v\in  C(K)} [\log  \delta^{v} (K) +b_d\int_K vd\mu]-b_d\int_K Qd\mu$.
\end{enumerate}
 Properties (1)-(4) also hold for the functionals $\underline J,\underline J^Q$. 	
\end{proposition}
\begin{proof}
Property (1) follows from
$$J^Q_n(G)\leq \sup_{{\bf a}\in \tilde G_{d_n}} |VDM^Q_n({\bf a})|^{1/l_n}\leq \sup_{{\bf a}\in K^{d_n}} |VDM^Q_n({\bf a})|^{1/l_n}.$$ 
The proofs of Corollary 3.4, Proposition 3.5 and Proposition 3.6 of \cite{PE} work mutatis mutandis to verify (2), (3) and (4). The relevant estimation, replacing the corresponding one which is two lines above equation (3.2) in \cite{PE}, is, given $\epsilon >0$, for ${\bf a}\in \tilde G_{d_n}$,
\begin{eqnarray}\label{POTA}
|VDM^Q_n({\bf a})|e^{\frac{nd_n}{\gamma_d}(-\epsilon-\int_K Qd\mu)} &\leq & |VDM_n({\bf a})|\\
& \leq & |VDM^Q_n({\bf a})|e^{\frac{nd_n}{\gamma_d}(\epsilon+\int_K Qd\mu)}.\nonumber
\end{eqnarray}
To see this, we first recall that
$$|VDM_n({\bf a})|=|VDM^Q_n({\bf a})|e^{n\sum_{j=1}^{d_n}Q(a_j)}.$$
For $\mu \in {\mathcal M}_P(K)$, $Q\in C(K)$, $\epsilon >0$, there exists a neighborhood $G$ of $\mu$ in ${\mathcal M}_P(K)$ with 
$$-\epsilon < \int_K Qd\mu - \frac{\gamma_d}{d_n}\sum_{j=1}^{d_n}Q(a_j)<\epsilon$$
for ${\bf a}\in \tilde G_{d_n}$. Plugging this double inequality into the previous equality we get (\ref{POTA}).
Moreover, from (\ref{key}),
\begin{equation} \label{key2}\lim_{n\to \infty} \frac{nd_n}{l_n} = \frac{d+1}{Ad}=b_d\gamma_d\end{equation}
so that $\frac{nd_n}{\gamma_d}\asymp l_nb_d$ as $n\to \infty$. Taking $l_n-$the roots in (\ref{POTA}) accounts for the factor of $b_d$ in (2), (3) and (4).
\end{proof}

\begin{remark} \label{for45} The corresponding $\underline W,\underline W^Q, \overline W,\overline W^Q$ functionals, defined using (\ref{wfcl}), clearly dominate their ``$J$'' counterparts; e.g., $\overline W^Q\geq \overline J^Q$.

\end{remark}

Note that formula (\ref{enrum}) can be rewritten:
\begin{equation}\label{enrum2} \log \delta^Q(K)= -b_dE(V_{P,K,Q}^*).\end{equation}
Thus the upper bound in Proposition \ref{prop: relation J JQ} $(3)$ becomes
\begin{equation}\label{jmue}
\log \overline J(\mu)\leq -b_d\sup_{v\in C(K)}[E(V^*_{P,K,v})-\int_Kvd\mu]=-b_d E^*(\mu).
\end{equation}

For the rest of section \ref{sec:eandp} and section \ref{sec:ld}, we will always assume $Q\in C(K)$. Theorem \ref{obsolete} shows that the inequalities in (3) and (4) are equalities, and that the $\overline J,\overline J^Q$ functionals coincide with their $\underline J,\underline J^Q$ counterparts. The key step in the proof of Theorem \ref{obsolete} is to verify this for $\overline J^v(\mu_{K,v})$ and $\underline J^v(\mu_{K,v})$.

 \begin{theorem} \label{obsolete} Let $K\subset \CC^d$ be a nonpluripolar compact set and let $\nu$ satisfy a strong Bernstein-Markov property. Fix $Q\in  C(K)$. Then for any $\mu\in \mathcal M_P(K)$, 
 \begin{equation}\label{minunwtd}\log \overline J(\mu)= \log \underline  J(\mu)=\inf_{v\in  C(K)} [\log  \delta^{v} (K) +b_d\int_K v d\mu]\end{equation}
and
 \begin{equation}\label{minwtd}\log \overline J^Q(\mu)= \log \underline J^Q(\mu)=\inf_{v\in  C(K)} [\log  \delta^{v} (K) +b_d\int_K v d\mu]-b_d\int_K Qd\mu.\end{equation}
 \end{theorem}
 \begin{proof} It suffices to prove (\ref{minunwtd}) since (\ref{minwtd}) follows from $(2)$ of Proposition \ref{prop: relation J JQ}. We have the upper bound 
 $$\log \overline J(\mu)\leq  \inf_{v\in  C(K)} [\log  \delta^{v} (K) +b_d\int_K vd\mu]$$
 from (3); for the lower bound, we consider different cases.

 \smallskip \noindent {\sl Case I: $\mu=\mu_{K,v}$ for some $v\in  C(K)$.}
\smallskip

 We verify that
\begin{equation}\label{jversion}\log \overline J(\mu_{K,v})=\log \underline J(\mu_{K,v})=\log \delta^{v} (K) +b_d\int_K vd\mu_{K,v}\end{equation}
which proves (\ref{minunwtd}) in this case. 
 \medskip
 
  {To prove (\ref{jversion}), we use the definition of $ \underline J(\mu_{K,v})$ and Corollary \ref{largedev}. Fix a neighborhood $G$ of $\mu_{K,v}$. For $\eta >0$, define $A_{n,\eta}$ as in (\ref{aketa}) with $Q=v$. Set 
 \begin{equation}
 	\label{etan}
 	\eta_n := \max \left( \delta^v(K)- \frac{nZ_n^{1/2l_n} }{n+1}, \frac{Z_n^{1/2l_n}}{n+1}\right).
 \end{equation}
 By  Proposition \ref{weightedtd}, $\eta_n\to 0$. 
 We claim that we have the inclusion
 \begin{equation}\label{setinclu} 
  A_{n,\eta_n} \subset \tilde G_{d_{n}} \  \hbox{for all} \ n\  \hbox{large enough}. 
 \end{equation}
We prove \eqref{setinclu} by contradiction: if false, there is a sequence $\{n_j\}$ with $n_j\uparrow \infty$  and  $x^j=(x_1^j,...,x_{d_{n_j}}^j)\in 
 A_{n_j,\eta_{n_j}} \setminus \tilde G_{d_{n_j}}$. However $\mu_j:=\frac{\gamma_d}{d_{n_j}}\sum_{i=1}^{d_{n_j}}\delta_{x_i^j}\not\in  G$ for $j$ sufficiently large contradicts Theorem \ref{asympwtdfek} since $x^j\in A_{n_j,\eta_j} $ and $\eta_j \downarrow 0$ imply $\mu_j\to \mu_{K,v}$ weak-*.

Next,  a direct computation using \eqref{etan} shows that, for all $n$ large enough,
 \begin{equation}\label{probk2} Prob_{n}(K^{d_{n}}\setminus A_{n,\eta_n})\leq \frac{(\delta^v(K)-\eta_n)^{2l_n}}{Z_n} \leq (\frac{n}{n+1})^{2l_n}
 \leq \frac{n}{n+1}
 \end{equation}
(recall $\nu$ is a probability measure). Hence 
\begin{flalign*}
	& \frac{1}{Z_{n}} \int_{\tilde G_{d_{n}}}  |VDM_{n}^{v}(z_1,...,z_{d_{n}
})|^2 \cdot d\nu(z_1) \cdots d\nu(z_{d_{n}
})\\
& \geq   \frac{1}{Z_{n}} \int_{A_{n,\eta_n}}  |VDM_{n}^{v}(z_1,...,z_{d_{n}
})|^2 \cdot d\nu(z_1) \cdots d\nu(z_{d_{n}
})\\
&\geq  \frac{1}{n+1}. 
\end{flalign*}
Since $P\subset r \Sigma$ and $\Sigma \subset kP$ for some $k\in \ZZ^+$,  $l_n = 0 (n^{d+1})$ and we have $\frac{1}{2l_n} \log (n+1) \to 0$.} 
Since $\nu$ satisfies a strong Bernstein-Markov property and $v\in C(K)$, using Proposition \ref{weightedtd}  and the above estimate we conclude that
$$\liminf_{n\to \infty} \frac{1}{2l_{n}}\log \int_{\tilde G_{d_{n}}}  |VDM_{n}^{v}(z_1,...,z_{d_{n}
})|^2 d\nu(z_1) \cdots d\nu(z_{d_{n}
})$$
$$\geq \log  \delta^{{v}}(K).$$
Taking the infimum over all neighborhoods $G$ of $\mu_{K,v}$ we obtain
$$\log \underline J^{v}(\mu_{K,v})\geq \log  \delta^{{v}}(K).$$
{From (1) Proposition \ref{prop: relation J JQ}, $\log \overline J^{v}(\mu_{K,v})\leq \log  \delta^{{v}}(K)$; thus we have} 
\begin{equation}\label{jeqn}\log \underline J^{v}(\mu_{K,v})=\log \overline J^{v}(\mu_{K,v})= \log  \delta^{{v}}(K).\end{equation}
Using (2) of Proposition \ref{prop: relation J JQ} with $\mu = \mu_{K,v}$ we obtain (\ref{jversion}).

\smallskip \noindent {\sl Case II: $\mu\in \mathcal M_P(K)$ with the property that $E^*(\mu)<\infty$.}
\smallskip

{ From Theorem \ref{thm: existence in E1 local} and Proposition \ref{pro: E1 characterization local} there exists $u\in L_P(\CC^d)$ -- indeed, $u\in \Ec^1_P(\CC^d)$ -- with $\mu =(dd^cu)^d$ and $\int_K ud\mu>-\infty$}.  However, since $u$ is only usc on $K$, $ \mu$ is not necessarily of the form $\mu_{K,v}$ for some $v\in  C(K)$. Taking a sequence of continuous functions $\{Q_j\}\subset  C(K)$ with $Q_j \downarrow u$ on $K$,  by Proposition \ref{goodapprox} the weighted extremal functions $V^*_{P,K,Q_j}$ decrease to $u$ on $\CC^d$;  
$$\mu_j:=(dd^cV^*_{P,K,Q_j})^d\to \mu=(dd^cu)^d \ \hbox{weak-}*;$$
and 
\begin{equation} \label{needit}
\lim_{j\to \infty} \int_K Q_j d\mu_j =\lim_{j\to \infty} \int_K Q_j d\mu=\int_K u d\mu.
\end{equation}
From the previous case we have 
$$\log \overline J(\mu_j)= \log  \underline J(\mu_j)=\log  \delta^{Q_j} (K) +b_d\int_K Q_jd\mu_j.$$
Using uppersemicontinuity of the functional $\mu \to \underline J(\mu)$,
$$\limsup_{j\to \infty} \overline J(\mu_j)=\limsup_{j\to \infty} \underline J(\mu_j)\leq \underline  J(\mu).$$
Since $Q_j \downarrow u$ on $K$, 
\begin{equation}\label{upbound}
\limsup_{j\to \infty} \log  \delta^{Q_j}(K)=\lim_{j\to \infty}  \log  \delta^{Q_j}(K).
\end{equation}
Therefore
$$M:=\lim_{j\to \infty} \log \underline J(\mu_j)=\lim_{j\to \infty}\bigl( \log  \delta^{Q_j} (K) +b_d\int_K Q_jd\mu_j \bigr)$$
exists and is less than or equal to $\log \underline J(\mu)$. 
We want to show that 
\begin{equation}\label{finbound}\inf_{v} [\log  \delta^{v} (K) +b_d\int_K vd\mu]\leq M.\end{equation}

Given $\epsilon >0$, by \eqref{needit} for $j\geq j_0(\epsilon)$,
$$\int_K Q_j d\mu_j \geq \int_K Q_j d\mu -\epsilon \ \hbox{and} \ \log \underline J(\mu_j)< M+\epsilon.$$
Hence for such $j$,
$$\inf_{v} [\log  \delta^{v} (K) +b_d\int_K vd\mu]\leq \log  \delta^{Q_j} (K) +b_d\int_K Q_jd\mu$$
$$\leq \log  \delta^{Q_j} (K) +b_d\int_K Q_jd\mu_j+b_d\epsilon=\log \underline J(\mu_j)+b_d\epsilon < M+(b_d+1)\epsilon,$$
yielding (\ref{finbound}). This finishes the proof in Case II.

\smallskip \noindent {\sl Case III: $\mu\in \mathcal M(K)$ with the property that $E^*(\mu)=+\infty$.}
\smallskip

It follows from Proposition \ref{gap}  and Theorem \ref{energyrumely} that the right-hand side of  \eqref{minunwtd} is $-\infty$, finishing the proof.

 \end{proof}

\begin{remark} \label{fourfour} From now on, we simply use the notation $J,J^Q$ without the overline or underline. 
Using Proposition \ref{gap} and Theorem \ref{energyrumely}, we have
$$\log  J(\mu)= \inf_{Q\in  C(K)} [\log  \delta^{Q} (K) +b_d\int_K Qd\mu]$$
$$=-\sup_{Q\in  C(K)} [-\log  \delta^{Q} (K) -b_d\int_K Qd\mu]$$
$$=-\sup_{Q\in  C(K)} [b_dE(V_{P,K,Q}^*) -b_d\int_K Qd\mu]=-b_d\sup_{Q\in  C(K)} [E(V_{P,K,Q}^*) -\int_K Qd\mu]$$
(recall (\ref{enrum2})) which one can compare with
$$E^*(\mu)=\sup_{Q\in C(K)}[E(V^*_{P,K,Q})-\int_K Q d\mu]$$
from Proposition \ref{gap} to conclude
\begin{equation}\label{eandj}\log  J(\mu)=-b_d E^*(\mu).\end{equation}
In particular, $J, \ J^Q$ are independent of the choice of strong Bernstein-Markov measure for $K$.
\end{remark}

Following the idea in Proposition 4.3 of \cite{PELD}, we observe the following:

\begin{proposition} \label{prop43} Let $K\subset \CC^d$ be a nonpluripolar compact set and let $\nu$ satisfy a strong Bernstein-Markov property. Fix $Q\in  C(K)$. The measure $\mu_{K,Q}$ is the unique maximizer of the functional $\mu \to J^Q(\mu)$ over $\mu \in \mathcal M_P(K)$; i.e., 
\begin{equation}\label{check}  J^Q(\mu_{K,Q})=\delta^Q(K)  \ (\hbox{and} \ J(\mu_K)=\delta(K)).
\end{equation}
 \end{proposition}
\begin{proof} The fact that $\mu_{K,Q}$ maximizes $J^Q$ (and $\mu_K$ maximizes $J$) follows from \eqref{jversion}, \eqref{jeqn} and Proposition \ref{prop: relation J JQ}.

Assume now that $\mu\in \mathcal{M}_P(K)$ maximizes $J^Q$.    From Remark \ref{for45} and the definitions of the functionals, for any neighborhood $G\subset \mathcal M_P(K)$ of $\mu$,  
$$\overline J^Q(\mu)\leq \overline W^Q(\mu)\leq \sup\{\limsup_{n\to \infty} |VDM_{n}^{Q}({\bf a}^{(n)})|^{1/l_n}\}\leq \delta^Q(K)$$ 
where the supremum is taken over all arrays $\{{\bf a}^{(n)}\}_{n=1,2,...}$ of $d_n-$tuples ${\bf a}^{(n)}$ in $K$ whose normalized counting measures $\mu_n:=\frac{1}{d_n}\sum_{j=1}^{d_n} \delta_{a^{(n)}_j}$ lies in $G$.  Since $\overline J^Q(\mu) = \delta^Q(K)$  there is an asymptotic weighted  
Fekete array $\{{\bf a}^{(n)}\}$ as in (\ref{wam}). Theorem \ref{asympwtdfek} yields that $\mu_n:=\frac{1}{d_n}\sum_{j=1}^{d_n} \delta_{a^{(n)}_j}$ converges weak-* to $\mu_{K,Q}$, hence $\mu_{K,Q}\in \bar{G}$.   Since this is true for each neighborhood $G\subset \mathcal M_P(K)$ of $\mu$, we must have $\mu=\mu_{K,Q}$. 
\end{proof}

\section{Large deviation.}\label{sec:ld} As in the previous section, we fix $K\subset \CC^d$ a nonpluripolar compact set; $Q\in  C(K)$; and a measure $\nu$ on $K$ satisfying a strong Bernstein-Markov property. For $x_1,...,x_{d_n}\in K$, we get a discrete measure $\frac{\gamma_d}{d_n}\sum_{j=1}^{d_n} \delta_{x_j}\in \mathcal M_P(K)$. Define 
$j_n:  K^{d_n} \to \mathcal M_P(K)$ via 
$$j_n(x_1,...,x_{d_n}):=\frac{\gamma_d}{d_n}\sum_{j=1}^{d_n} \delta_{x_j}.$$
From (\ref{probk}), 
$\sigma_n:=(j_n)_*(Prob_n) $ is a probability measure on $\mathcal M_P(K)$: for a Borel set $B\subset \mathcal M_P(K)$,
\begin{equation}\label{sigmak}\sigma_n(B)=\frac{1}{Z_n} \int_{\tilde B_{d_n}} |VDM_n^Q(x_1,...,x_{d_n
})|^2 d\nu(x_1) \cdots d\nu(x_{d_n
})\end{equation}
where $\tilde B_{d_n}:= \{{\bf a} =(a_1,...,a_{d_n})\in K^{d_n}: \frac{\gamma_d}{d_n}\sum_{j=1}^{d_n} \delta_{a_j}\in B\}$(recall (\ref{nbhddef})). Here, $Z_n:=Z_n(P,K,Q,\nu)$. Note that 
\begin{equation}\label{jandsigma} \sigma_n(B)^{1/2l_n}=\frac{1}{Z_n^{1/2l_n}}\cdot J_n^Q(B).\end{equation}
For future use, suppose we have a function $F:\RR\to \RR$ and a function $v\in C(K)$. We write, for $\mu \in \mathcal M_P(K)$, 
$$<v,\mu
>:= \int_K v d\mu
$$
and then 
\begin{equation}\label{lambdaint}\int_{\mathcal M_P(K)} F(<v,\mu
>)d\sigma_n(\mu
):=\end{equation}
$$\frac{1}{Z_n} \int_K \cdots \int_K  |VDM_n^Q(x_1,...,x_{d_n
})|^2  F\left(\frac{\gamma_d}{d_n
}\sum_{j=1}^{d_n
} v(x_j)\right)d\nu(x_1) \cdots d\nu(x_{d_n
}).$$

With this notation, we offer two proofs of our LDP, Theorem \ref{ldp}. We state the result; define LDP in Definition \ref{equivform}; and then proceed with the proofs. This closely follows the exposition in section 5 of \cite{PELD}.

\begin{theorem} \label{ldp} The sequence $\{\sigma_n=(j_n)_*(Prob_n)\}$ of probability measures on $\mathcal M_P(K)$ satisfies a {\bf large deviation principle} with speed $2l_n$ and good rate function $\mathcal I:=\mathcal I_{K,Q}$ where, for $\mu \in \mathcal M_P(K)$,
$$\mathcal I(\mu):=\log J^Q(\mu_{K,Q})-\log J^Q(\mu).$$
 \end{theorem}

This means that $\mathcal I:\mathcal M_P(K)\to [0,\infty]$ is a lowersemicontinuous mapping such that the sublevel sets $\{\mu \in \mathcal M_P(K): \mathcal I(\mu)\leq \alpha\}$ are compact in the weak-* topology on $\mathcal M_P(K)$ for all $\alpha \geq 0$ ($\mathcal I$ is ``good'') satisfying (\ref{lowb}) and (\ref{highb}):

\begin{definition} \label{equivform}
The sequence $\{\mu_k\}$ of probability measures on $\mathcal M_P(K)$ satisfies a {\bf large deviation principle} (LDP) with good rate function $\mathcal I$ and speed $2l_n$ if for all 
measurable sets $\Gamma\subset \mathcal M_P(K)$, 
\begin{equation}\label{lowb}-\inf_{\mu \in \Gamma^0}\mathcal I(\mu)\leq \liminf_{n\to \infty} \frac{1}{2l_n} \log \mu_n(\Gamma) \ \hbox{and}\end{equation}
\begin{equation}\label{highb} \limsup_{n\to \infty} \frac{1}{2l_n} \log \mu_n(\Gamma)\leq -\inf_{\mu \in \bar \Gamma}\mathcal I(\mu).\end{equation}
\end{definition}

In the setting of $\mathcal M_P(K)$, to prove a LDP it suffices to work with a base for the weak-* topology. The following is a special case of a basic general existence result for a LDP given in Theorem 4.1.11 in \cite{DZ}.

\begin{proposition} \label{dzprop1} Let $\{\sigma_{\epsilon}\}$ be a family of probability measures on $\mathcal M_P(K)$. Let $\mathcal B$ be a base for the topology of $\mathcal M_P(K)$. For $\mu\in \mathcal M_P(K)$ let
$$\mathcal I(\mu):=-\inf_{\{G \in \mathcal B: \mu \in G\}}\bigl(\liminf_{\epsilon \to 0} \epsilon \log \sigma_{\epsilon}(G)\bigr).$$
Suppose for all $\mu\in \mathcal M_P(K)$,
$$\mathcal I(\mu)=-\inf_{\{G \in \mathcal B: \mu \in G\}}\bigl(\limsup_{\epsilon \to 0} \epsilon \log \sigma_{\epsilon}(G)\bigr).$$
Then $\{\sigma_{\epsilon}\}$ satisfies a LDP with rate function $\mathcal I(\mu)$ and speed $1/\epsilon$. 
\end{proposition}

There is a converse to Proposition \ref{dzprop1}, Theorem 4.1.18 in \cite{DZ}. For $\mathcal M_P(K)$, it reads as follows:

\begin{proposition} \label{dzprop2} Let $\{\sigma_{\epsilon}\}$ be a family of probability measures on $\mathcal M_P(K)$. Suppose that $\{\sigma_{\epsilon}\}$ satisfies a LDP with rate function $\mathcal I(\mu)$ and speed $1/\epsilon$. Then for any base $\mathcal B$ for the topology of $\mathcal M_P(K)$  and any $\mu\in \mathcal M_P(K)$ 
$$\mathcal I(\mu):=-\inf_{\{G \in \mathcal B: \mu \in G\}}\bigl(\liminf_{\epsilon \to 0} \epsilon \log \sigma_{\epsilon}(G)\bigr)$$
$$=-\inf_{\{G \in \mathcal B: \mu \in G\}}\bigl(\limsup_{\epsilon \to 0} \epsilon \log \sigma_{\epsilon}(G)\bigr).$$ 
\end{proposition}

\begin{remark} \label{equival} Assuming Theorem \ref{ldp}, this shows that, starting with a strong Bernstein-Markov measure $\nu$ and the corresponding sequence of probability measures $\{\sigma_n\}$ on $\mathcal M_P(K)$ in (\ref{sigmak}), the existence of 
an LDP with rate function $\mathcal I(\mu)$ and speed $2l_n$ implies that
necessarily
\begin{equation} \label{ourrate} \mathcal I(\mu)=\log J^Q(\mu_{K,Q})-\log J^Q(\mu).\end{equation} 
Uniqueness of the rate function is basic (cf., Lemma 4.1.4 of \cite{DZ}). \end{remark}

We turn to the first proof of Theorem \ref{ldp}, using Theorem \ref{obsolete}, which gives a pluripotential theoretic description of the rate functional. 

\begin{proof} As a base $\mathcal B$ for the topology of $\mathcal M_P(K)$, we can take the sets from (\ref{nbhdbase}) or simply all open sets. For $\{\sigma_{\epsilon}\}$, we take the sequence of probability measures $\{\sigma_n\}$ on $\mathcal M_P(K)$ and we take $\epsilon =\frac{1}{2l_n}$. For $G\in \mathcal B$, from (\ref{jandsigma}),
$$\frac{1}{2l_n}\log \sigma_n(G)= \log J_n^Q(G)-\frac{1}{2l_n}\log Z_n.$$
From Proposition \ref{weightedtd}, and (\ref{jeqn}) with $v=Q$, 
$$\lim_{n\to \infty} \frac{1}{2l_n}\log Z_n=\log  \delta^Q(K)= \log J^Q(\mu_{K,Q});$$ and by Theorem \ref{obsolete}, 
$$\inf_{G \ni \mu} \limsup_{n\to \infty} \log J_n^Q(G)=\inf_{G \ni \mu} \liminf_{n\to \infty} \log J_n^Q(G)=\log J^Q(\mu).$$

Thus by Proposition \ref{dzprop1}  $\{\sigma_n\}$ satisfies an LDP with rate function 
$$\mathcal I(\mu):=\log J^Q(\mu_{K,Q})-\log J^Q(\mu)$$
and speed $2l_n$. This rate function is good since $\mathcal M_P(K)$ is compact.
\end{proof}

\begin{remark} \label{hope} From Proposition \ref{prop43}, $\mu_{K,Q}$ is the unique maximizer of the functional 
$$\mu \to \log J^Q(\mu) $$
over all $\mu\in {\mathcal M}_P(K)$. Thus
$$\mathcal I_{K,Q}(\mu)\geq 0 \ \hbox{with} \ \mathcal I_{K,Q}(\mu)= 0 \iff \mu=\mu_{K,Q}. $$
To summarize, $\mathcal I_{K,Q}$ is a good rate function with unique minimizer $\mu_{K,Q}$. Using the relations 
$$\log  J(\mu)= -b_d\sup_{Q\in  C(K)} [E(V_{P,K,Q}^*) -\int_K Qd\mu]$$
$$J(\mu)= J^Q(\mu)\cdot (e^{\int_K Qd\mu})^{b_d}, \ \hbox{and} \ J^Q(\mu_{K,Q})=\delta^Q(K)$$
(the latter from (\ref{check})), we have
$$\mathcal I(\mu):=\log \delta^Q(K)-\log J^Q(\mu)$$
$$=\log \delta^Q(K)-\log J(\mu)+b_d\int_K Qd\mu$$
$$=b_d\sup_{Q\in  C(K)} [E(V_{P,K,Q}^*) -\int_K Qd\mu]+\log \delta^Q(K)+b_d\int_K Qd\mu$$
$$=b_d\sup_{v\in  C(K)} [E(V_{P,K,v}^*)-\int_K vd\mu]- b_d[E(V_{P,K,Q}^*)-\int_K Qd\mu]$$
from (\ref{enrum2}).  
\end{remark}

The second proof of our LDP follows from Corollary 4.6.14 in \cite{DZ}, which is a general version of the G\"artner-Ellis theorem. This approach was originally brought to our attention by S. Boucksom and was also utilized by R. Berman in \cite{Ber}. We state the version of the \cite{DZ} result for an appropriate family of probability measures.

\begin{proposition} \label{gartell}
Let $C(K)^*$ be the topological dual of $C(K)$, and let $\{\sigma_{\epsilon}\}$ be a family of probability measures on ${\mathcal M}_P(K)\subset C(K)^*$ (equipped with the weak-* topology). Suppose for each $\lambda \in C(K)$, the limit
$$\Lambda(\lambda):=\lim_{\epsilon \to 0} \epsilon \log \int_{C(K)^*} e^{\lambda(x)/\epsilon}d\sigma_{\epsilon}(x)$$
exists as a finite real number and assume $\Lambda$ is G\^ateaux differentiable; i.e., for each $\lambda, \theta\in C(K)$, the function $f(t):= \Lambda(\lambda +t\theta)$ is differentiable at $t=0$. Then $\{\sigma_{\epsilon}\}$ satisfies an LDP in $C(K)^*$ with the convex, good rate function $\Lambda^*$.
\end{proposition}

Here $$\Lambda^*(x):= \sup_{\lambda \in C(K)}\bigl(<\lambda,x>- \Lambda(\lambda)\bigr),$$ is the Legendre transform of $\Lambda$. The upper bound (\ref{highb}) in the LDP holds with rate function $\Lambda^*$ under the assumption that the limit $\Lambda(\lambda)$ exists and is finite; the G\^ateaux differentiability of $\Lambda$ is needed for the lower bound (\ref{lowb}). To verify this property in our setting, we must recall a result from \cite{BBL}. 

\begin{proposition}
\label{diffpropcor}
For  $Q\in \mathcal A(K)$ and $u\in C(K)$, let 
$$F(t):=E(V_{P,K,Q+tu}^*)$$
for $t\in \RR$. Then $F$ is differentiable and
$$
F'(t)= \int_{\CC^d} u (dd^c V_{P,K,Q+tu}^*)^d.
$$
\end{proposition}

\noindent In \cite{BBL} it was assumed that $u\in C^2(K)$ but the result is true with the weaker assumption $u\in C(K)$ (cf., Theorem 11.11 in \cite{GZ} due to Lu and Nguyen \cite{LN}, see also \cite[Proposition 4.20]{DDL}).

We proceed with the second proof of Theorem \ref{ldp}. For simplicity, we normalize so that $\gamma_d =1$ to fit the setting of Proposition \ref{gartell} (so members of  ${\mathcal M}_P(K)$ are probability measures).

\begin{proof} We show that for each $v\in C(K)$,
$$\Lambda(v):=\lim_{n \to \infty} \frac{1}{2l_n
}\log  \int_{C(K)^*} e^{2l_n
<v,\mu>}d\sigma_{n}(\mu)$$ exists as a finite real number. First, since $\sigma_n$ is a measure on ${\mathcal M}_P(K)$, the integral can be taken over ${\mathcal M}_P(K)$. Consider
$$\frac{1}{2l_n
}\log  \int_{{\mathcal M}_P(K)} e^{2l_n
<v,\mu>}d\sigma_{n}(\mu).$$
By (\ref{lambdaint}), this is equal to
$$\frac{1}{2l_n
}\log  \frac{1}{Z_n}\cdot \int_{K^{d_n
}}  |VDM_n^{Q-\frac{l_n}{nd_n}v}(x_1,...,x_{d_n
})|^2 d\nu(x_1) \cdots d\nu(x_{d_n
}).$$
From (\ref{key2}), with $\gamma_d=1$, $\frac{l_n}{nd_n}\to \frac{1}{b_d}$; hence for any $\epsilon >0$, 
$$\frac{1}{b_d+\epsilon}v \leq \frac{l_n}{nd_n}v\leq \frac{1}{b_d-\epsilon}v \ \hbox{on} \ K$$
for $n$ sufficiently large. Recall that
$$Z_n= \int_{K^{d_n
}} |VDM_n^Q(x_1,...,x_{d_n
}))|^2d\nu(x_1) \cdots d\nu(x_{d_n
}).$$
Define
$$\tilde Z_n:= \int_{K^{d_n
}}  |VDM_n^{Q-v/b_d}(x_1,...,x_{d_n
})|^2 d\nu(x_1) \cdots d\nu(x_{d_n
}).$$
Then we have
$$\lim_{n\to \infty} \tilde Z_n^{\frac{1}{2l_n
}}=  \delta^{Q-v/b_d} (K) \ \hbox{and} \ \lim_{n\to \infty}  Z_n^{\frac{1}{2l_n
}}=  \delta^Q (K) $$
from (\ref{zeen}) in Proposition \ref{weightedtd} and the assumption that $(K,\nu,\tilde Q)$ satisfies the weighted Bernstein-Markov property for {\it all} $\tilde Q\in  C(K)$. Thus
\begin{equation}\label{lambda}\Lambda(v)= \lim_{n \to \infty} \frac{1}{2l_n
}\log \frac{\tilde Z_n}{Z_n}= \log \frac{ \delta^{Q-v/b_d} (K)}{  \delta^{Q} (K)} .\end{equation}
Define now, for $v, v'\in C(K)$, 
$$f(t):=  E(V^*_{P,K,Q-(v+tv')}).$$
Proposition \ref{diffpropcor} shows that $\Lambda$ is G\^ateaux
 differentiable and Proposition \ref{gartell} gives that $\Lambda^*$ is a rate function on $C(K)^*$.
 
 Since each $\sigma_n$ has support in $\mathcal M_P(K)$, it follows from (\ref{lowb}) and (\ref{highb}) in Definition \ref{equivform} of an LDP with $\Gamma \subset C(K)^*$ that for $\mu\in C(K)^*\setminus \mathcal M_P(K)$, $\Lambda^*(\mu)=+\infty$. By Lemma 4.1.5 (b) of \cite{DZ}, the restriction of $\Lambda^*$ to $\mathcal M_P(K)$ is a rate function. Since $\mathcal M_P(K)$ is compact, it is a good rate function. Being a Legendre transform, $\Lambda^*$ is convex.  

To compute $\Lambda^*$, we have, using (\ref{lambda}) and (\ref{enrum}),
$$\Lambda^*(\mu)= \sup_{v\in C(K)} \bigl( \int_K v d\mu -  \log \frac{ \delta^{Q-v/b_d} (K)}{  \delta^{Q} (K)}\bigr)$$
$$=\sup_{v\in C(K)} \bigl( \int_K v d\mu - b_d[E(V^*_{P,K,Q})-E(V^*_{P,K,Q-v/b_d}])\bigr).$$
Thus $$\Lambda^*(\mu)+b_dE(V^*_{P,K,Q})= \sup_{v\in C(K)} \bigl( \int_K v d\mu +b_dE(V^*_{P,K,Q-v/b_d})\bigr)$$
 $$= \sup_{u\in C(K)} \bigl( b_dE(V^*_{P,K,Q+u})-b_d\int_K u d\mu \bigr) \ (\hbox{taking} \ u=-v/b_d).$$
Rearranging and replacing $u$ in the supremum by $v=u+Q$,
$$\Lambda^*(\mu)=  \sup_{u\in C(K)} \bigl( b_dE(V^*_{P,K,Q+u})-b_d\int_K u d\mu \bigr)-b_dE(V^*_{P,K,Q})$$
$$= b_d\bigl[\sup_{v\in C(K)} E(V^*_{P,K,v})-\int_K v d\mu \bigr]-b_d\bigl[E(V^*_{P,K,Q})-\int_K Q d\mu\bigr]$$
which agrees with the formula in Remark \ref{hope} (since $\mu$ is a probability measure).

\end{proof}

\begin{remark} Thus the rate function can be expressed in several equivalent ways:
$$\mathcal I(\mu)= \Lambda^*(\mu)=\log J^Q(\mu_{K,Q})-\log J^Q(\mu)$$
$$= b_d\bigl[\sup_{v\in C(K)} E(V^*_{P,K,v})-\int_K v d\mu \bigr]-b_d\bigl[E(V^*_{P,K,Q})-\int_K Q d\mu\bigr]$$
$$=b_d E^*(\mu)-b_d\bigl[E(V^*_{P,K,Q})-\int_K Q d\mu\bigr]$$
which generalizes the result equating (5.3), (5.10) and (5.11) in \cite{PELD} for the case $P=\Sigma$ and $b_d=1$. Note in the last equality we are using the slightly different notion of $E^*$ in (\ref{estar2}) and Proposition \ref{gap} than that used in \cite{PELD}.

\end{remark}

 \end{document}